\documentclass[twocolumn]{article}

\usepackage[utf8]{inputenc}
\usepackage[american]{babel}
\usepackage{enumerate,paralist}
\usepackage{amsfonts,amsmath,amssymb,amsthm,amstext,latexsym}	
\usepackage{url}
\usepackage{xspace}
\usepackage{array}
\usepackage{ifthen}
\usepackage{intcalc}
\usepackage{algorithm}
\usepackage{algpseudocode}
\usepackage{subfig}
\usepackage{fullpage}

\usepackage[pdftex]{graphicx}
\graphicspath{{fig/}}

\usepackage{todonotes}

\theoremstyle{plain}
\newtheorem{theorem}{Theorem}

\newcommand{\LapSmo}{\textsc{Laplacian Smoothing}\xspace}
\newcommand{\ori}{\texttt{ORI}\xspace}
\newcommand{\bfs}{\texttt{BFS}\xspace}
\newcommand{\rdr}{\texttt{RDR}\xspace}

\begin{document}
\title{Locality-Aware Laplacian Mesh Smoothing}

\author{Guillaume Aupy\\
Electrical Eng. and Computer Science\\
Vanderbilt University\\
Nashville, TN, USA\\
guillaume.aupy@vanderbilt.edu\\
   \and
JeongHyung Park \\
Computer Science and Engineering\\
Penn State University\\
University Park, PA, USA\\
jxp975@cse.psu.edu\\
   \and
Padma Raghavan\\
Electrical Eng. and Computer Science\\
Vanderbilt University\\
Nashville, TN, USA\\
guillaume.aupy@vanderbilt.edu\\   }

\maketitle

\begin{abstract}
In this paper, we propose a novel reordering scheme to improve the
performance of a Laplacian Mesh Smoothing (LMS). While the Laplacian smoothing
algorithm is well optimized and studied, we show how a simple reordering of the
vertices of the mesh can greatly improve the execution time of the smoothing
algorithm. The idea of our reordering is based on (i) the postulate that cache
misses are a very time consuming part of the execution of LMS, and (ii) the
study of the reuse distance patterns of various executions of the LMS algorithm. 

Our reordering algorithm is very simple but allows for huge performance
improvement. We ran it on a Westmere-EX platform and obtained a speedup of 75 on
32 cores compared to the single core execution without reordering, and a gain in
execution of 32\% on 32 cores compared to state of the art reordering.
Finally, we show that we leave little room for a better ordering by reducing
the L2 and L3 cache misses to a bare minimum.
\end{abstract}


\section{Introduction}
In HPC systems, when solving partial differential equations using finite element
methods for unstructured meshes in parallel on multicore system, high quality
meshes are required~\cite{DR:Mellor}. This is due to the fact that mesh quality
plays a significant role in both the accuracy of the solution produced by a PDE
solver, as well as the solver's execution time~\cite{MQ:Shewchuk}. However, when
attempting to parallelize the mesh smoothing application to obtain solution of
PDE efficiently, we do not obtain full performance scalability due to the
irregular memory accesses of the application~\cite{MQ:Zhou2010}.

When doing a computation over a graph, the vertices are stored in a
data-structure according to a specific ordering. Strout and
Hovland~\cite{MQ:Strout} noticed that the data-ordering of irregular HPC
applications impacted the performance of these applications. In particular, they
showed that a breadth-first search reordering heuristic outperformed all
existing ordering heuristics for mesh smoothing~\cite{MQ:Strout}.

\begin{figure}[h!]
  \centering
  \subfloat[Random ordering, average reuse distance: 90k, L1 cache miss rate: 2.18\%, execution time: 10.28s.]
	{\includegraphics[width=0.72\linewidth]{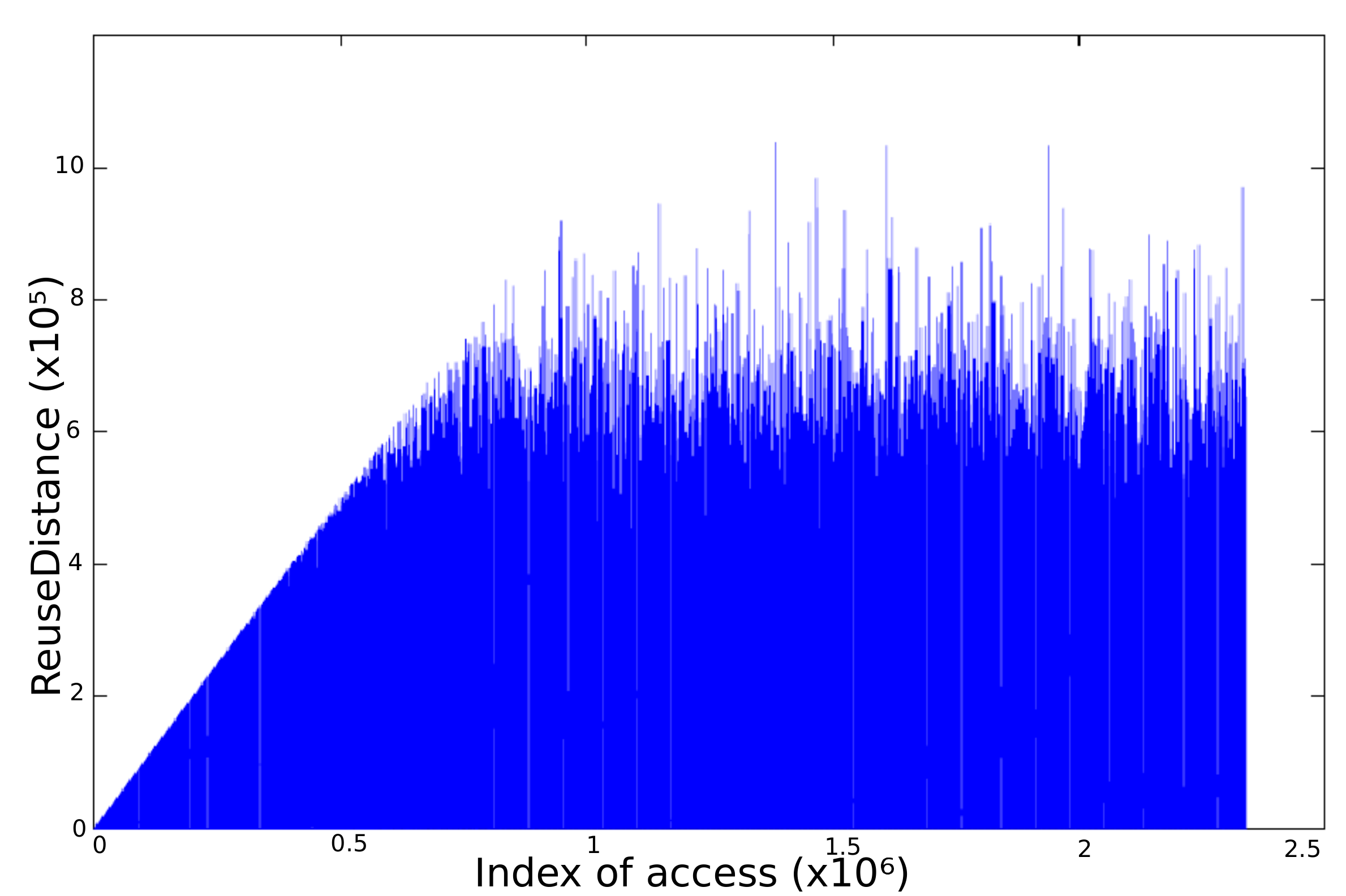}}\\
  \subfloat[Original ordering, average reuse distance: 4450, L1 cache miss rate: 0.71\%, execution time: 7.6s.]
	{\includegraphics[width=0.72\linewidth]{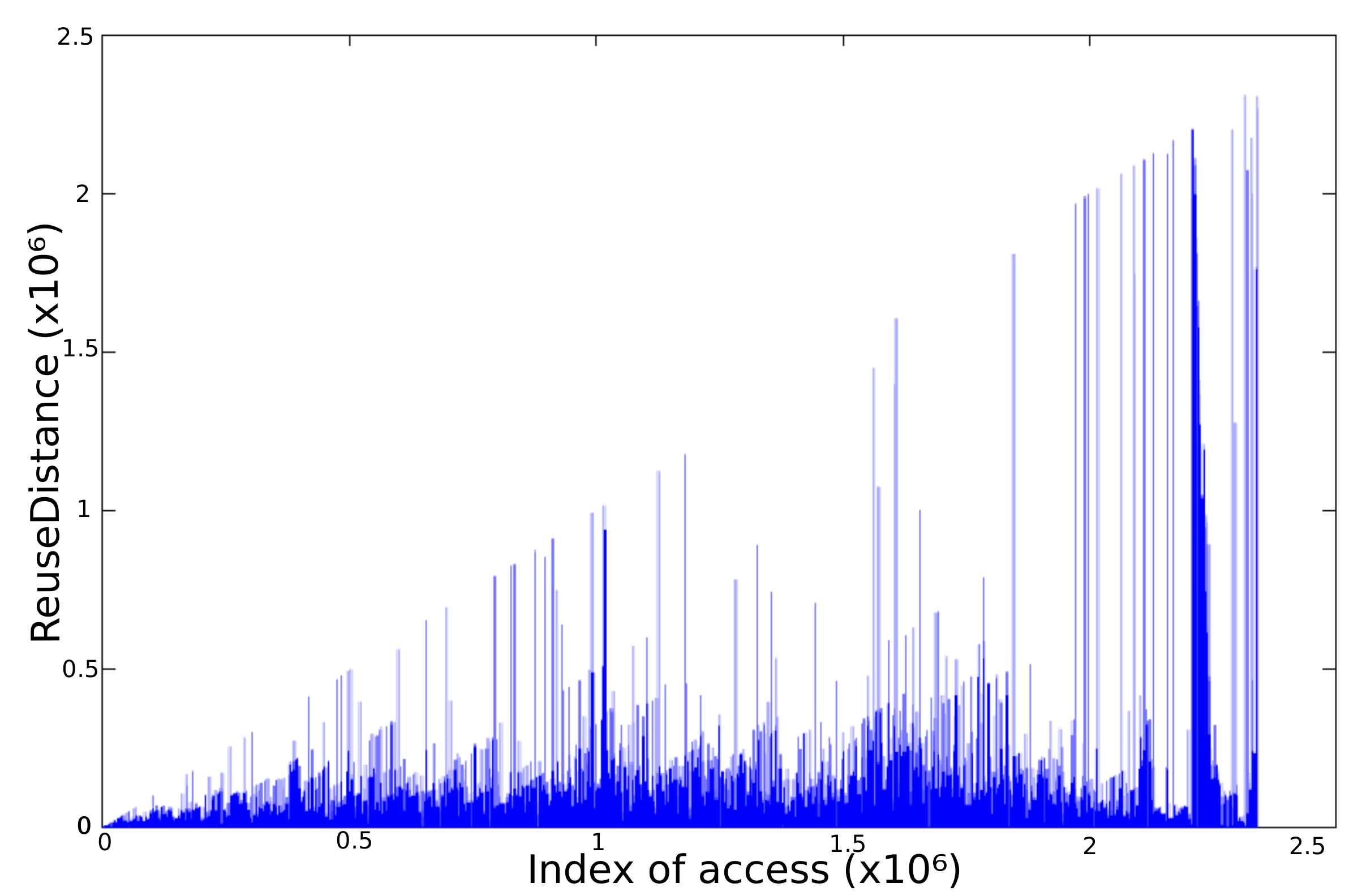}}\\
  \subfloat[\bfs ordering, average reuse distance: 2910, L1 cache miss rate: 0.59\%, execution time: 6.59s.]
	{\includegraphics[width=0.72\linewidth]{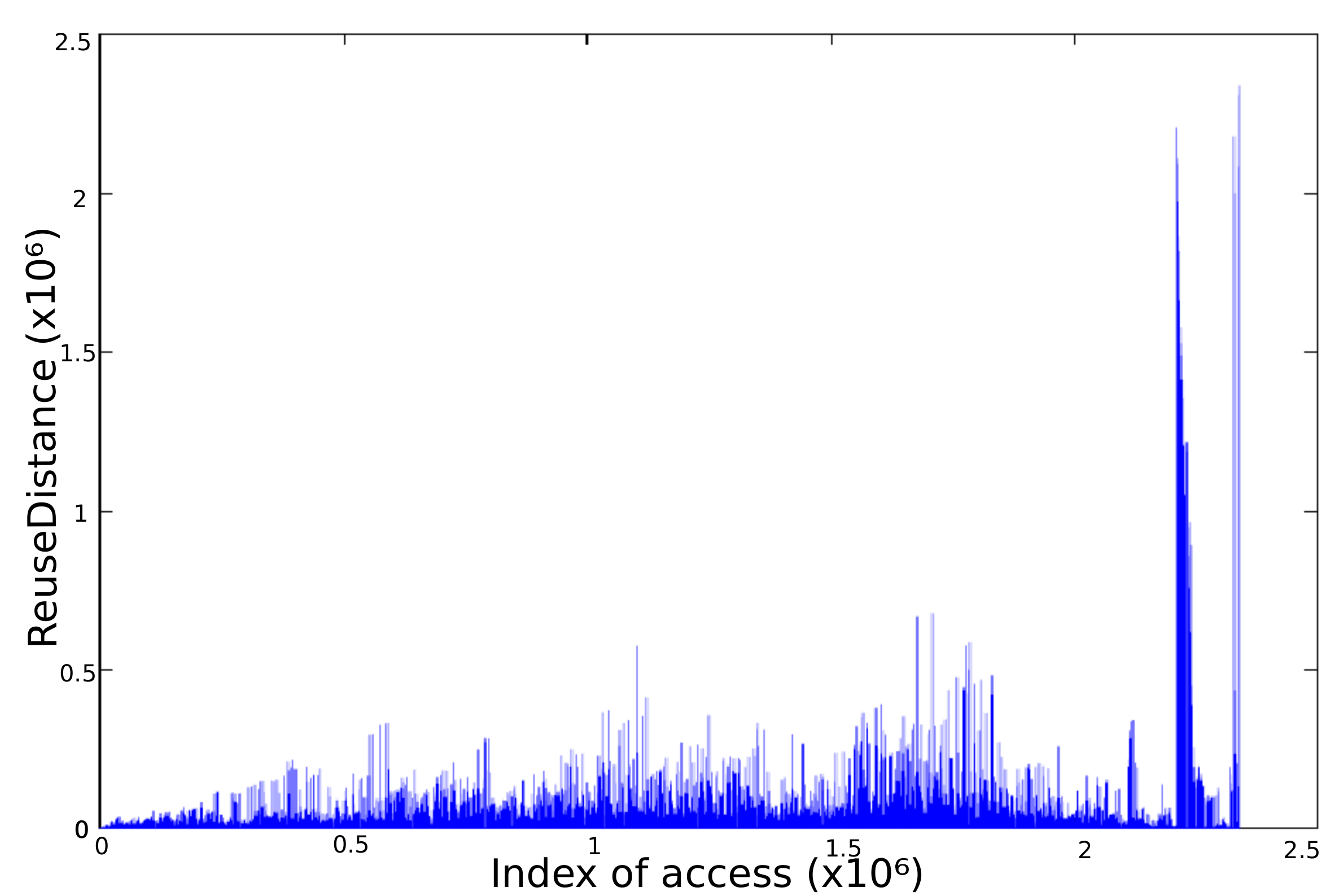}}
\caption{Reuse distance plays a significant role in cache performance of HPC applications
on multicore~\cite{RD:Beyls}. Simply put, the reuse distance is the number of
distinct data accesses between two consecutive accesses of the same data element.
We plot here the reuse distance of the first iteration of the LMS algorithm on
an {\em ocean} mesh for different orderings.}\label{RDreorderexe}
\vspace{-1cm}
\end{figure}

On an identical execution, different orderings have different effects
on the reuse distance and on the cache miss rate. The performance is impacted.
In Figure~\ref{RDreorderexe}, we show how different ordering impacts the average
reuse distance, the execution time and the cache misses of an execution
of LMS. In particular, we compare three different orderings: a random ordering
(that unsurprisingly has the worse performance), the original ordering of the
mesh (given by the mesh creation algorithm~\cite{shewchuk96b}), and 	the ordering
of the \bfs algorithm introduced by Strout and Hovland~\cite{MQ:Strout}. This
strongly supports the importance of choosing a good ordering.

In this work, we propose a novel ordering algorithm that outperforms that of
Strout and Hovland~\cite{MQ:Strout}. We give some insights on the reason of its
performance by studying thoroughly the reuse distance and cache misses of this
algorithm on various meshes. 
Finally, while our algorithm is cache-oblivious as it focuses in reducing the
reuse distance independently of the cache size, we show that it reduces the L2
and L3 cache misses to a bare minimum on the platform we ran it on hence making
this algorithm quasi-optimal amongst the possible reordering algorithms.

The rest of this paper is organized as follows. We start by presenting related 
work on locality-aware iterative algorithms in Section~\ref{sec:related}. We
then introduce the problematic of reuse distance and cache miss
(Section~\ref{sec:reuse}) and the Laplacian Mesh Smoothing algorithm
(Section~\ref{sec:lms}). Then in Section~\ref{sec:mesh_reorder} we describe our
novel reordering scheme. Finally, we provide and discuss experimental results in
Section~\ref{sec:expe}. Section~\ref{sec:conclusion} provides conclusions and 
direction for future work.

\section{Related work}
	\label{sec:related}
Reordering schemes have proven very efficient to improve performance of many
irregular iterative applications such as Feasible Newton mesh
optimization~\cite{imr,MQ:Strout,feasible}, mesh
warping~\cite{sastry2014parallel} or Laplacian Mesh
Smoothing~\cite{park,imr,MQ:Strout,feasible}. In this context, Munson and
Hovland~\cite{feasible} developed a Feasible Newton mesh optimization algorithm
and benchmark. Their algorithm and benchmark employed both data and iteration
reordering in order to improve cache performance. One finding of their research
was that reordering of the input data can increase or decrease the number of
iterations taken by the inexact Newton method and can affect its success or
failure~\cite{feasible}. They compared six different reorderings of the vertices
and elements in the mesh. They showed that a breadth-first search of the
vertices, followed by a reversing of the order in which the vertices were
visited performed better than the original ordering. When data and iteration
ordering were performed on the relevant hypergraphs, the reorderings were found
to significantly decrease the number of cache misses in all phases of code
execution and resulted in significantly faster code~\cite{MQ:Strout,feasible}.
In this work, we compare our results to their results.

Recently, Shontz and Knupp~\cite{imr} considered mesh vertex reordering
techniques to reduce the total time required to improve the mesh quality. Vertex
ordering was performed for the first iteration (static) and every iterations
(dynamic). Their main finding was that static strategies were superior to
dynamic ones because of the overhead of the additional reorderings. Then they
compared twenty different orderings from the literature and were not able to
find an ordering that stood out as an all-purpose ordering. Park, Knupp and
Shontz~\cite{park} confirmed this result. Based on their result, this work
focuses on an \emph{a priori} ordering. We design a novel ordering that not only
stands out for all meshes considered and outperforms existing heuristics.
Unlike this study, we consider cache performance when a reordering technique is
applied to mesh smoothing.

\section{Problem definitions}
	\label{sec:defs}
In this Section, we introduce the basis for the different notions introduced in
this paper.

\subsection{Reuse distance and cache misses}
	\label{sec:reuse}
Current architectures have different levels of storage systems. They are what
are called \emph{hierarchical architectures}: each level has different access
costs (as a rule of thumb, the closer the storage system, the faster the access)
and different storage sizes (again, often the closest systems have the smallest
size).

When a piece of data is accessed, it is stored to the closest level of storage.
Then when other data are used, it is pushed back to higher level of storage
where the access costs are larger. When a computation tries to access a piece of
data and that data is not stored on the closest storage, we call this a
\emph{miss}. Then the computation has to go and fetch it from a remote storage
which induces additional costs.

In this work we focus on NUMA architectures. In this model, there are four
different storage systems: three caches (L1, L2, L3) and a main memory. They are
organized with the L1 cache being the smallest and fastest storage system and
memory being the largest and slowest storage system. 
\begin{figure}[!h]
\centering
\includegraphics[width=0.7\linewidth]{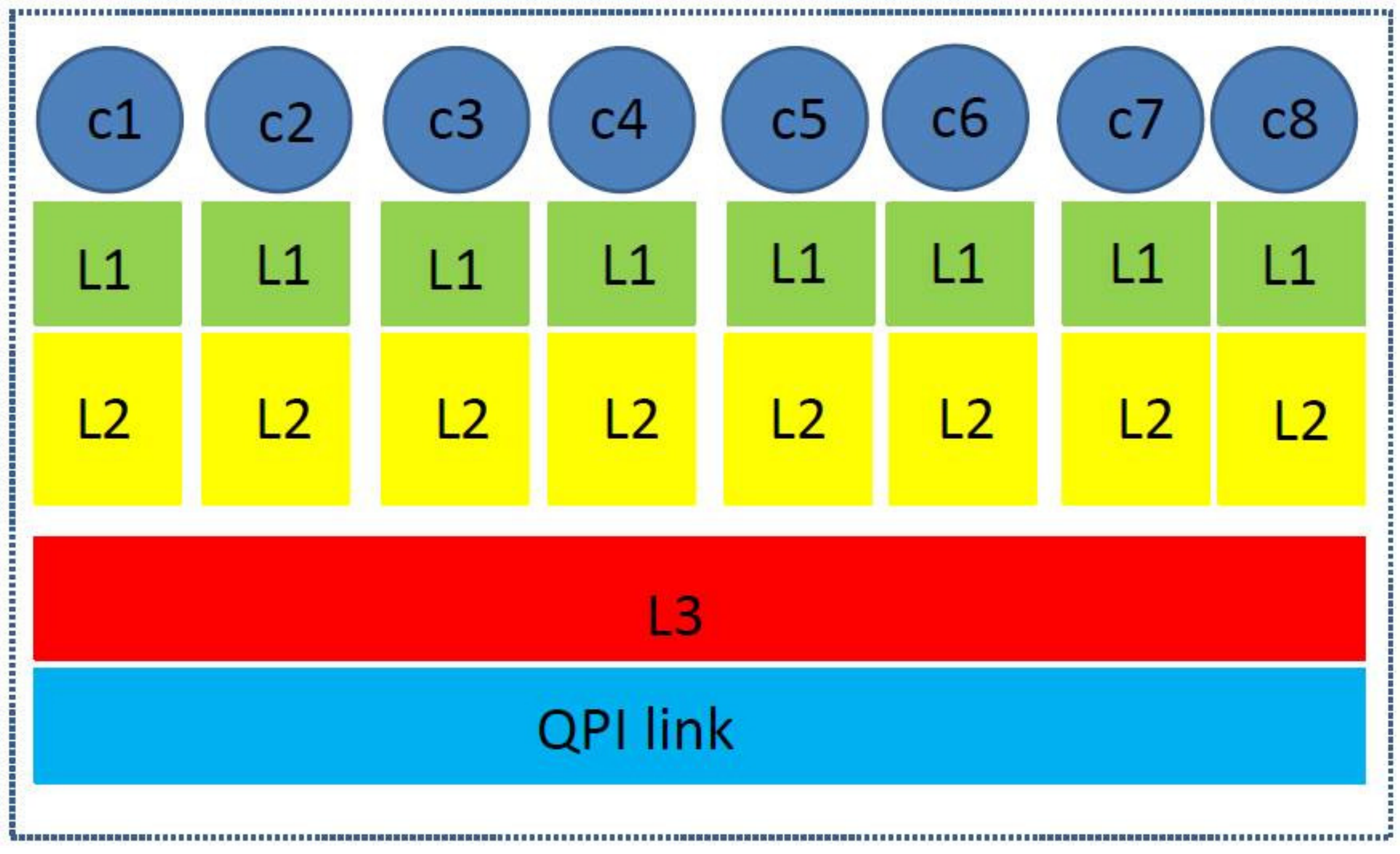}
\caption{High-level view of a socket of the Intel Westmere-EX processor. The
machine has four sockets, which are connected directly via 3.2 GHz QPI links.
Each socket has eight cores $c_1$ to $c_8$ with the inclusive cache hierarchy;
32K L1 private cache, 256K L2 private cache, and 24MB shared L3 cache.
\label{architecture}}
\end{figure}

To choose which unused data is being pushed back to remote storage systems, the
caches often follow the Least Recently Used (LRU) algorithm: the piece of data
that was the least recently used is the one that will be pushed to remote
storage. With this algorithm is introduced the notion of \emph{reuse distance}:
between two accesses to the same data, what are the number of distinct piece of
data that were accessed. 

Intuitively, if the reuse distance is greater than the size of the L1 cache,
there will be a L1-cache miss. If it is greater than the size of L2, there will
also be an additional L2-cache miss, etc.

In practice the behavior is slightly different~\cite{drepper2007every}. For
instance, the fetching is done by cache lines (multiple elements at once) and
not by elements which impacts the actual cache misses. However, in this paper we
use this simpler theoretical model as a first order approximation to understand
and give an intuition of the results observed.

\subsection{Laplacian Mesh smoothing}
	\label{sec:lms}
\begin{algorithm}
\caption{\textit{Algorithm for Laplacian Mesh Smoothing}}\label{algo:mesh_smooth}
\begin{algorithmic}[1]
\Procedure{\LapSmo}{$V$, $T$}
\State $V \gets$ mesh vertex data
\State $T \gets$ mesh triangle data
\State quality $\gets 0$
\For{$i=1$ to $|V|$}
    \State compute $q_{V[i]}$: mesh quality for $V[i]$
    \State increase quality by $q_{V[i]}$
\EndFor
\State $\texttt{Global\_quality} = \frac{1}{|V|}$quality
\While{$\texttt{Global\_quality} <$ goal quality \label{algoline:mesh_quality}}
	\For{$i\in \{$ interior vertices $\}$}
	    \State Smooth $V[i]$ using Equation~\eqref{eq:smooth}
		\State Update $\texttt{Global\_quality}$
	\EndFor
\EndWhile
\EndProcedure
\end{algorithmic}
\end{algorithm}

Mesh smoothing application is performed to improve the quality of the mesh so
that an accurate PDE solution can be obtained within a short execution
time~\cite{fkms2002}. In the mesh smoothing procedure, the algorithm first
computes the initial mesh quality for a given mesh and do mesh smoothing to
improve the mesh quality~\cite{alternating}. 
After smoothing, we compute the mesh quality again and if the overall mesh
quality reached a desirable level, then we stop smoothing. Note that because the
desirable quality might never be attained, there is often a maximum number of
iterations set to the main loop (Algorithm~\ref{algo:mesh_smooth},
line~\ref{algoline:mesh_quality}). Note that exact details of the implementation
used can be found in the Mesquite software package~\cite{mesquite}.

We use edge-length ratio~\cite{MQ:Knupp2001}, i.e., the ratio of minimum and
maximum length edges, as a mesh quality metric for computing the mesh quality in
this study. The mesh quality for each vertex can be represented as an average
quality metric value of triangles that are attached on the vertex. The mesh
quality for the entire region of the mesh can be computed by averaging all mesh
quality values obtained from each vertex.
The range of edge-length ratio mesh quality values is $0 \sim 1$.
If the quality value for a triangle is close to $1$, we can say the triangle has
a good shape, i.e., is close to an equilateral triangle.
The goal of the mesh smoothing algorithm is to maximize the average quality
values for each vertex.

To improve the quality of the mesh, we perform Laplacian smoothing to replace
a vertex position using neighbor vertices coordinates. Suppose there is a
vertex $v$ we want to move and $N$ neighbored vertices surrounding the vertex.
If we represent the position for $i^{th}$ neighbored vertex as $p_{i}$, the new
position for $\overline{p}_{v}$ will be
\begin{equation}
    \overline{p}_{v} = \frac{1}{N}\sum^{N}_{i=1}p_{i}. \label{eq:smooth}
\end{equation}
Figure~\ref{Lap} shows the initial mesh and the output mesh of Laplacian
smoothing.

\begin{figure}[h!]
\centering
\includegraphics[width=0.6\linewidth]{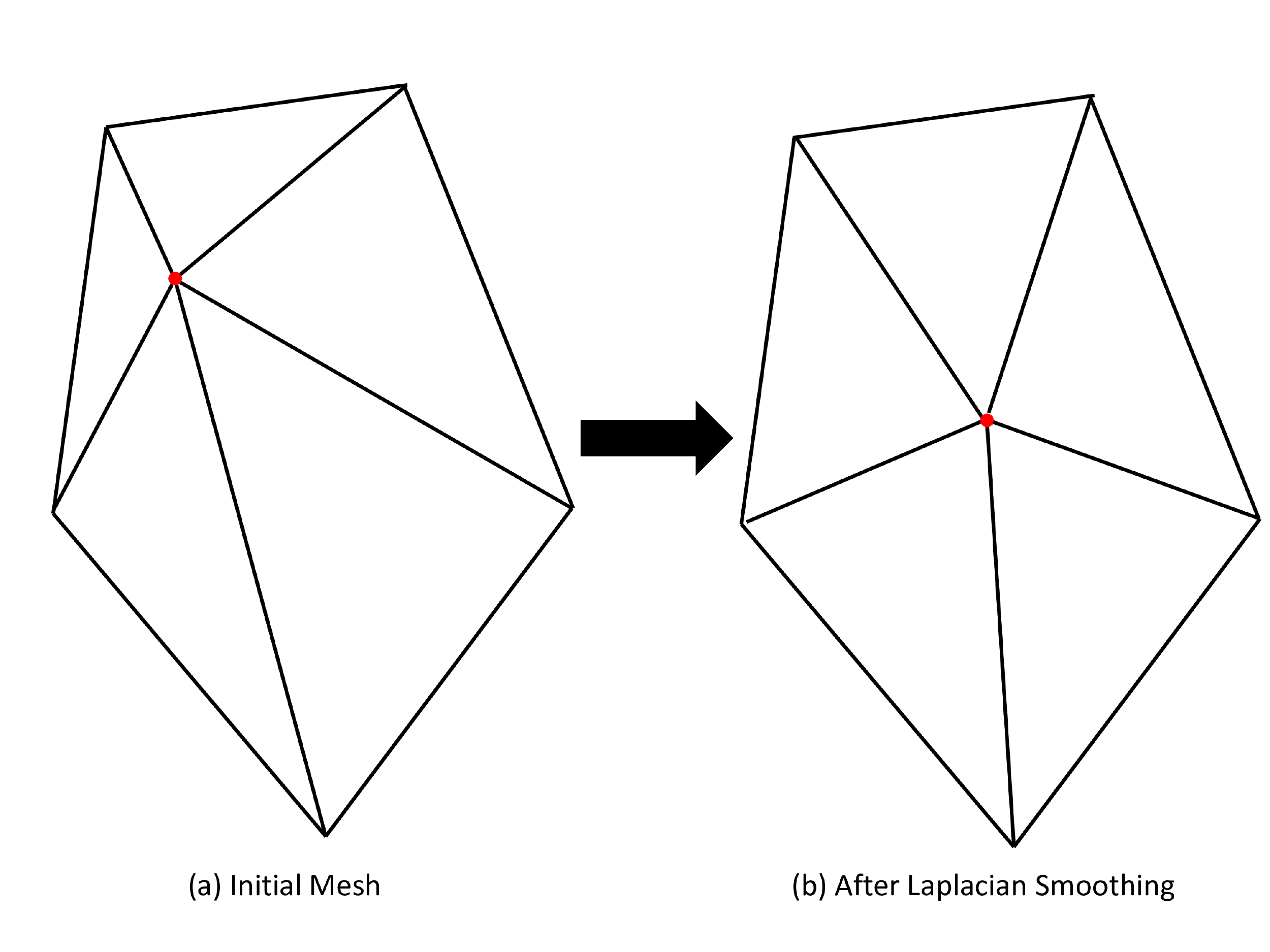}
\caption{Laplacian smoothing performed on initial mesh. The vertex position inside the mesh was changed.}
\label{Lap}
\end{figure}

We would like to improve the performance of the Laplacian mesh smoothing by
reordering the initial mesh. In order to do so, we reorder each vertex based on
the initial mesh quality for each vertex. Details are described in
Section~\ref{sec:mesh_reorder}.

\section{Mesh Reordering to Improve Mesh Smoothing Performance}
	\label{sec:mesh_reorder}

\subsection{Factors Affecting Temporal and Spatial Locality}
There are two factors that affect the performance of Laplacian mesh
smoothing~\cite{MQ:Strout}.

\emph{Spatial locality} is defined as the reuse within a cache line. The nodes
of a mesh are stored in the memory. When a node is selected, the node is
streamed to the cache along with its neighboring nodes (as many as can fit in a
cache line).
\emph{Temporal locality} is defined as the reuse of a node that is already
stored in the cache before it's cache line is discarded.

In both cases, the absence of sufficient data locality causes a cache miss, thus
adversely impacting the execution time.

Since the Laplacian smoothing method processes a node and some of its neighbors
successively, prefetching the neighboring nodes into the caches improves the
application's cache performance. In particular this is designated to improve the
temporal locality: since processing a node requires information from its
neighbors, it seems natural to process its neighbors then while the information
is still cached.
A good reordering strategy then aims at improving the spatial locality.

Intuitively, when one executes a node, one needs to fetch its neighbors, hence
having them close by together should improve spatial locality.
In this context, while naive, \bfs seems to be a good start to improve spatial
locality.

Figure~\ref{trace} shows two partial traces of node visiting observed from
Laplacian mesh smoothing when two different reuse distances are applied. When
the reuse distances for Laplacian mesh smoothing are reduced by applying \bfs
ordering, temporal and spatial localities are significantly improved. This
example dictates the interplay between temporal and spatial localities on one
hand and reducing reuse distances on the other.

\begin{figure}[!h]
  \centering
  \subfloat[Traces for Laplacian mesh smoothing with DFS ordering.]
  {\includegraphics[width=\linewidth]{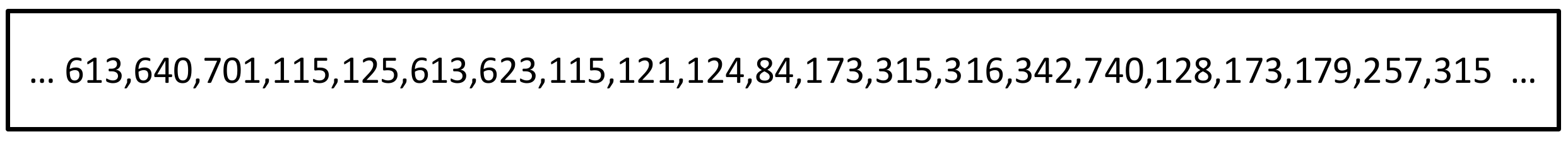}}\\
  \subfloat[Traces for Laplacian mesh smoothing with BFS ordering]
  {\includegraphics[width=\linewidth]{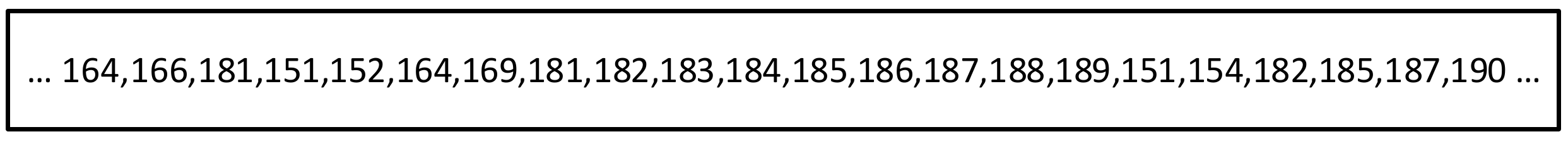}}
  \caption{Partial traces observed from node visiting for Laplacian mesh
  smoothing. The number represent the location of the data accessed in the data
  array. The closer the numbers, the shorter the distance between the accesses.}
  \label{trace}
\end{figure}

\subsection{Toward a Reuse Distance Reducing Ordering}
We now would like to consider how to reduce the reuse distances for Laplacian
mesh smoothing. 

The main idea of doing a reordering of the vertices is to improve locality when
they are accessed by the algorithm.
By the time the Laplacian mesh smoothing terminates, there is a certain order in
which nodes have been considered. 

We give an example of the usefulness of a good ordering in
Figure~\ref{fig:fake}. Consider the synthetic mesh shown in
Figure~\ref{fig:fake}.
Let us consider two orderings: the Depth First Search ordering given in
Figure~\ref{fakeori} and the Breadth First Search ordering given in
Figure~\ref{fakebfs}. Assume that node $j$ (10 in DFS, 3 in BFS) has worse
quality. Then during the execution of the algorithm (Read Data array), it will
be accessed first, followed by its neighbors: $k$, $m$, $i$, $a$, $b$ (to
compute its Laplacian value). In the DFS ordering, the range of accesses in the node array varies between
positions $1$ and $10$, while in the BFS ordering, it varies between positions
$1$ and $7$. Minimizing the span of accesses allows for a better spatial
locality.

\begin{figure}[h!]
  \centering
  \subfloat[DFS ordered Mesh.]
	{\includegraphics[width=0.8\linewidth]{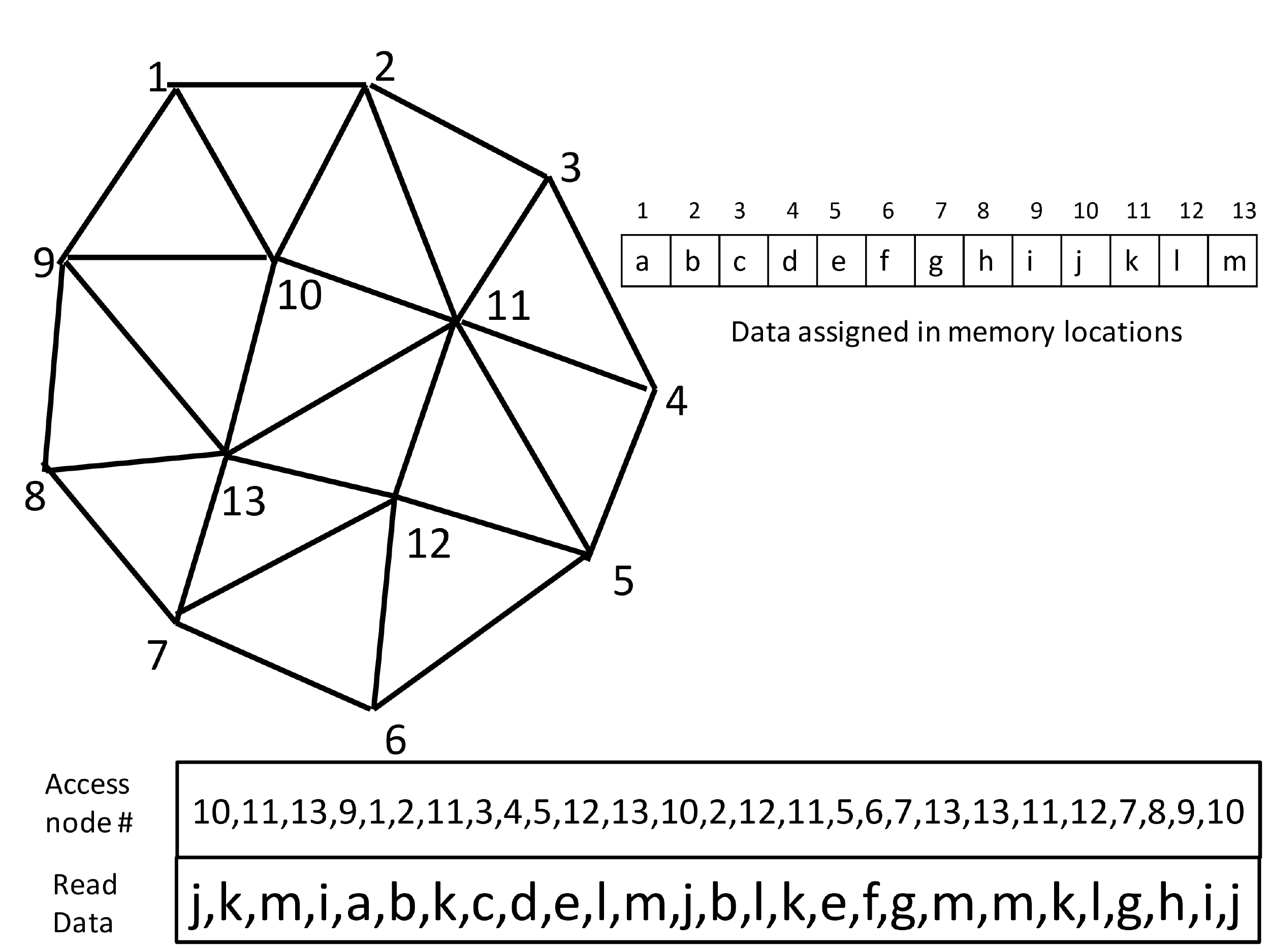}\label{fakeori}}\\
  \subfloat[BFS ordered Mesh.]
	{\includegraphics[width=0.8\linewidth]{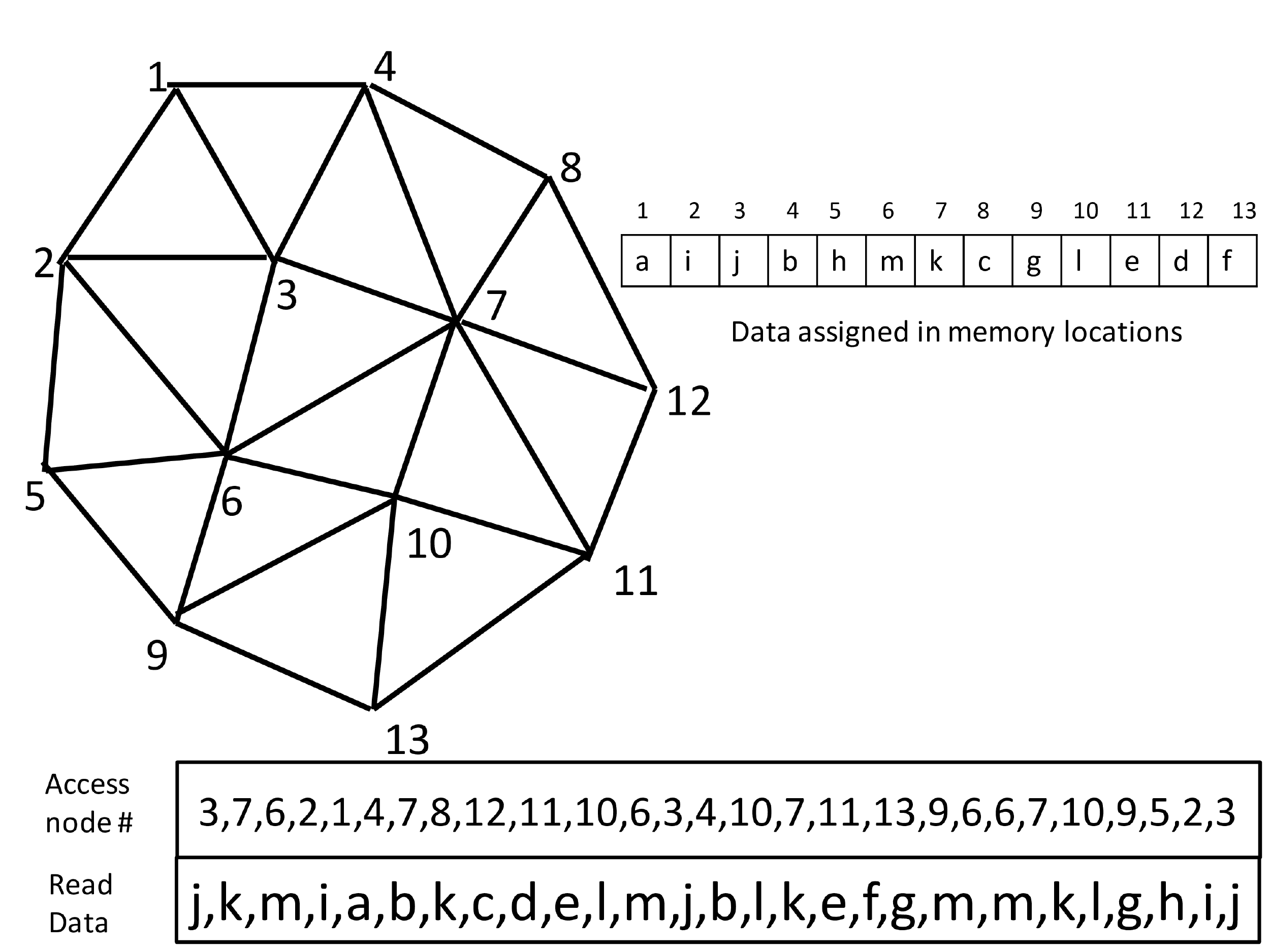}  \label{fakebfs}}\\
   \caption{For both mesh, the data values ($a$ to $m$) of each nodes are stored
in a 13 consecutive memory locations. 
While the ordering changes the location of the different data values in the data
array, the smoothing algorithm is identical (see the sequence of read data).
\label{fig:fake}}
\end{figure}

The nodes in the mesh reordered by BFS ordering store their nodes spatially.
This causes the access patterns for Laplacian mesh smoothing to become similar
to the memory access of nodes streamed in the cache. 
Though the accessed node list becomes similar to the streamed node list, there
are still spaces for improvement for obtaining optimal reuse distance of
Laplacian mesh smoothing.

We now consider improving temporal locality for Laplacian mesh smoothing.

The Laplacian mesh smoothing is a greedy algorithm.
When smoothing the mesh, the LMS algorithm starts by visiting the node that has
the worse quality. Once the smoothing process for the node is over, it selects
another node that has the worst quality among nodes nearby the node, i.e.,
neighboring nodes.

\begin{figure}[!h]
  \centering
\includegraphics[width=0.9\linewidth]{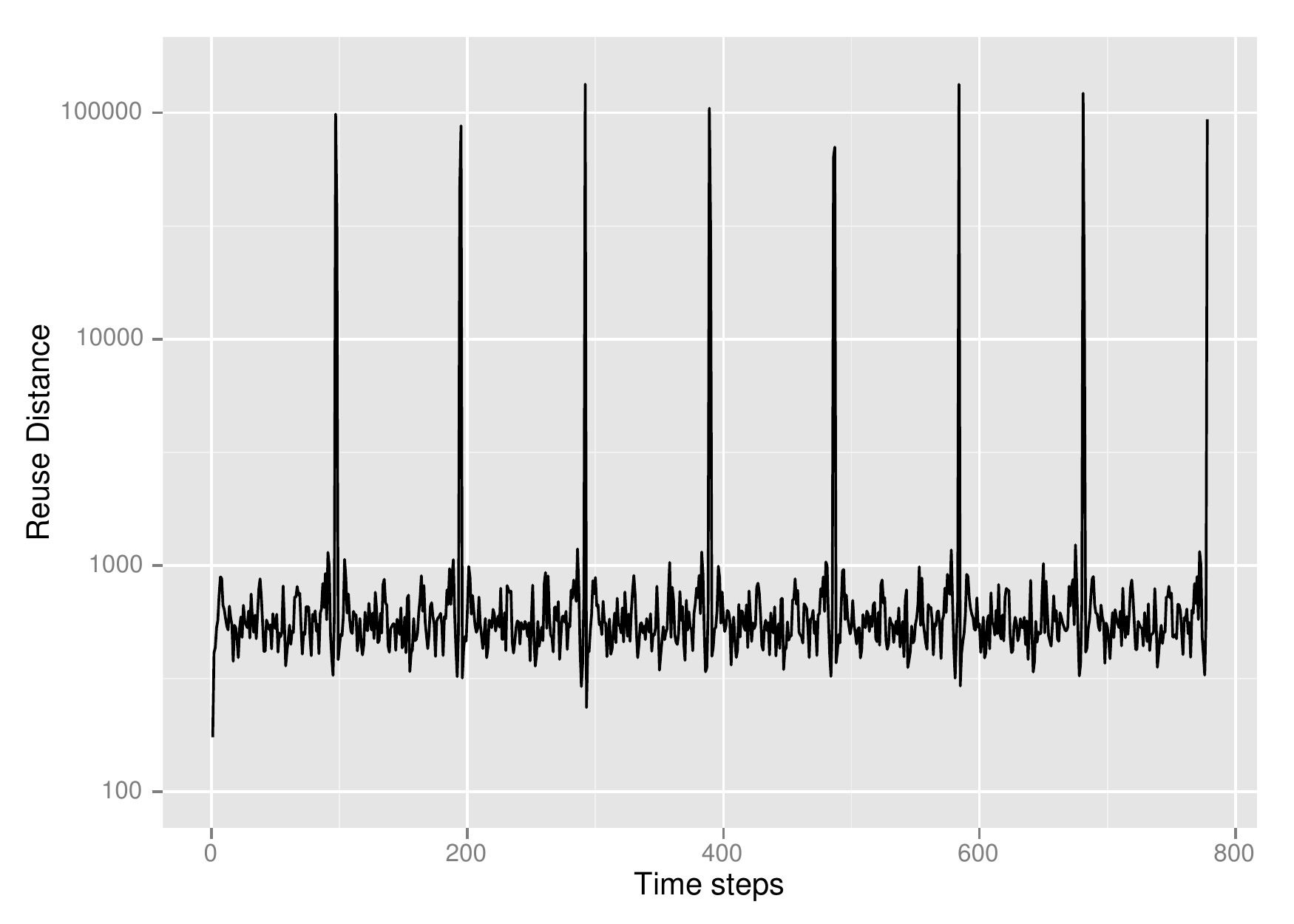}
\caption{Observed reuse distance profiles for a \emph{carabiner} mesh given the
initial ordering. In order to make it more readable, we divided each iteration
into 100 {\em Time Steps} where each time step is the average of approximatively
20,000 consecutive data accesses.\label{change}}
\end{figure}

We studied reuse distance patterns of the different iterations of the mesh
smoothing algorithm. 
As can be seen on Figure~\ref{change} (note that we had very similar results for
other meshes), the reuse distance has similar patterns over the different
iterations (there are eight iterations in the execution plotted on
Figure~\ref{change}). We conjecture that the access patterns
for Laplacian mesh smoothing can be controlled by the initial qualities of each
node in the mesh. 

Hence we conjecture that if we sort the nodes and their neighboring nodes based
on the qualities they have, the temporal locality will be improved. We base our
reordering heuristic on this conjecture.

Under this conjecture, we propose a mesh reordering scheme called \rdr in order
to reduce the reuse distance of Laplacian mesh smoothing.
Our reordering scheme (Algorithm~\ref{algo:distance_reorder}) follows a similar
pattern to the mesh smoothing algorithm iterations. Starting from the node with
the worse quality,
\begin{enumerate}
	\item From a given node already ordered, we sort all its neighbors that have
	not been ordered yet by increasing quality.
	\item We append them as such to the list of already ordered neighbors. 
	\item Then we perform the same algorithm for its neighbor of worse quality
(that has not been processed yet).
\end{enumerate}

\begin{algorithm}
\small
\caption{\textit{Algorithm for Reuse Distance Reducing Ordering}}\label{algo:distance_reorder}
\begin{algorithmic}[1]
\Procedure{\rdr}{$V$, $T$}
	\State $V_\texttt{new} \gets $ empty array of size $|V|$
	\State $\texttt{processed} \gets $ bool array of size $|V|$ initialized with false
	\State $\texttt{sorted} \gets $ bool array of size $|V|$ initialized with false
	\State \texttt{next\_num} $\gets 1$

	\For{$i\in \{$interior vertices sorted by increasing quality$\}$}
		\If{not $\texttt{processed}[i]$} \label{line:processed}
		    \If{not $\texttt{sorted}[i]$} 
		    	\State $V_\texttt{new}[\texttt{next\_num}] \gets V[i]$; incr \texttt{next\_num} 
		    	\State $\texttt{sorted}[i] \gets $ True \label{line:sort1}
		    \EndIf
		    \State $\texttt{processed}[i] \gets$ True \label{line:proc1}
			\State $l\gets\{$unprocessed neighbors of $i$ sorted by increasing quality$\}$

			\While{$l\neq \emptyset$}
				\For{$j=1$ to $|l|$}
					\If{not $\texttt{sorted}[l[j]]$}
					    \State $V_\texttt{new}[\texttt{next\_num}] \gets V[l[j]]$
					    \State incr \texttt{next\_num}
					    \State $\texttt{sorted}[l[j]]\gets$ True \label{line:sort2}
					\EndIf
			    \EndFor
			    \State $\texttt{processed}[l[0]] \gets$ True \label{line:proc2}
			    \State $l\gets\{$unprocessed neighbors of $l[0]$ sorted by increasing quality$\}$
			\EndWhile

		\EndIf
	\EndFor

	\State \Return $V_\texttt{new}$
\EndProcedure
\end{algorithmic}
\end{algorithm}

\begin{theorem}
Given a mesh $(V,T)$, then Algorithm~\ref{algo:distance_reorder} orders all
elements of the mesh exactly once.
\end{theorem}
\begin{proof}
The fact that each element of the mesh is ordered at most once is guaranteed by
the array \texttt{sorted}: once element $V[i]$ is added to $V_\texttt{new}$,
then the value $\texttt{sorted}[i]$ is set to True and never set back
to False. Furthermore, it can only be added to $V_\texttt{new}$ if 
$\texttt{sorted}[i]=\text{False}$.

The fact that each element is ordered at least once is guaranteed by the
property that (i) at the end of the algorithm, $\forall i\in \{$interior
vertices$\}$, $\texttt{processed}[i]=\text{True}$ (line~\ref{line:processed}
guarantees it), and (ii) $\texttt{processed}[i] \implies \texttt{sorted}[i]$
(every time $\texttt{processed}[i]$ is set to True, $\texttt{sorted}[i]$ was
also set to True earlier (line \ref{line:proc1}$\rightarrow$line \ref{line:sort1},
line \ref{line:proc2}$\rightarrow$line \ref{line:sort2})).
\end{proof}

When mesh smoothing application is executed, the application calculates the
initial qualities of each vertex in a mesh, smooths the mesh vertices, and then
computes the quality of the mesh to see the difference between initial and final
quality of the mesh.
We find that a vertex which has a bad quality in the beginning requires more
time to be smoothed compared to other vertices. If such vertices are processed
earlier than other vertices, the neighbors of the vertex which has the worst
quality are already in the cache and thus, we can improve the spatial locality.
The neighboring vertices are also reordered in increasing order.

By reordering all the vertices based on the methodology we described above, we
are able to obtain a node trace which is very similar to the initial streamed
node list for Laplacian mesh smoothing. Hence, we make sure that both temporal
and spatial localities are improved.

\section{Experimental Results}
	\label{sec:expe}
In this section, we evaluate our Reuse Distance Reducing ordering (\rdr) for
improving the performance of Laplacian mesh smoothing. We first describe the
experimental setup used in this study. We then show the execution times for
Laplacian mesh smoothing reduced by our reuse distance reducing ordering. We
also test the scalability of the Laplacian mesh smoothing with our reuse
distance reducing ordering.

\subsection{Experimental Setup}
	\label{sec:setup}
\noindent\textbf{System Setup. }
We used an Intel Westmere-EX architecture system to evaluate our reuse distance
reducing ordering. The Intel Westmere-EX architecture is equipped with 4
eight-core Intel Xeon E7-8837 processors. It supports 32 concurrent threads.
Each core has 32K L1 private cache and 256K L2 private cache with reported
access latencies of 4 and 10 cycles respectively~\cite{molka2009memory}, and
they share 24MB L3 cache. The L3 cache serves as the central unit for on-chip
inter-core accesses and accesses to off-chip processors include latencies of
data transfer on the QPI link. Consequently, L3 data access latencies can vary
from 38 to 170 cycles depending on core location and cache-line
state~\cite{molka2009memory}. The machine is an inclusive cache hierarchy. Each
processor is directly connected to other three processors via 3.2 GHz QPI links.
Additionally, access to the memory can range from 175 to 290
cycles~\cite{molka2009memory}. Figure~\ref{architecture} shows the high-level
view of the Intel Westmere-EX processor.

For multi-thread running, OpenMP library is used for both systems. Thread
affinity is set via KMP\_AFFINITY=compact, granularity=fine for pinning each
thread to each core. In all cases, thread scheduling is set to be static for
simply collecting the application trace of each thread by evenly dividing the
vertices. We implemented parallel Laplacian mesh smoothing based on the module
in Mesquite~\cite{mesquite}. For the purpose of this evaluation we put a quality 
convergence criterion to 0.000005 (meaning if the quality has improved by less
than this criterion, the execution stops). {\bf Note that the orderings did not
change the number of iterations needed to reach this criterion}

\noindent\textbf{Test Suite. }
To determine the impact of reuse distance reducing ordering on the mesh smoothing process, we tested nine meshes shown in Figure~\ref{meshes} (Coarse approximations are shown).
The meshes were generated by Triangle~\cite{shewchuk96b} and Table~\ref{meshinfo} gives their configurations.

\begin{figure}[!h]
  \centering
  \subfloat[carabiner]
  {\includegraphics[width=0.45\linewidth]{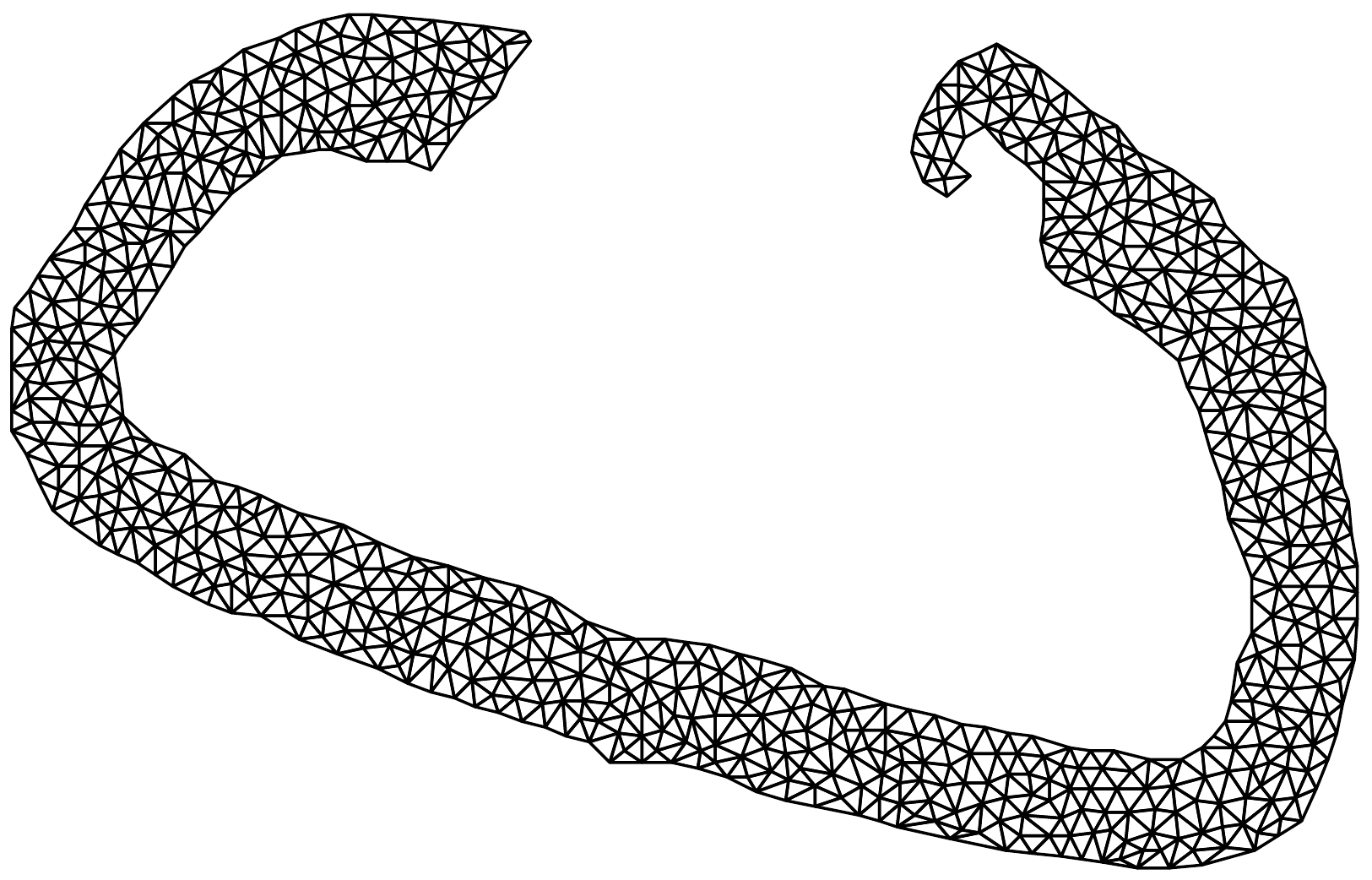}}~
\subfloat[crake]
  {\includegraphics[width=0.45\linewidth]{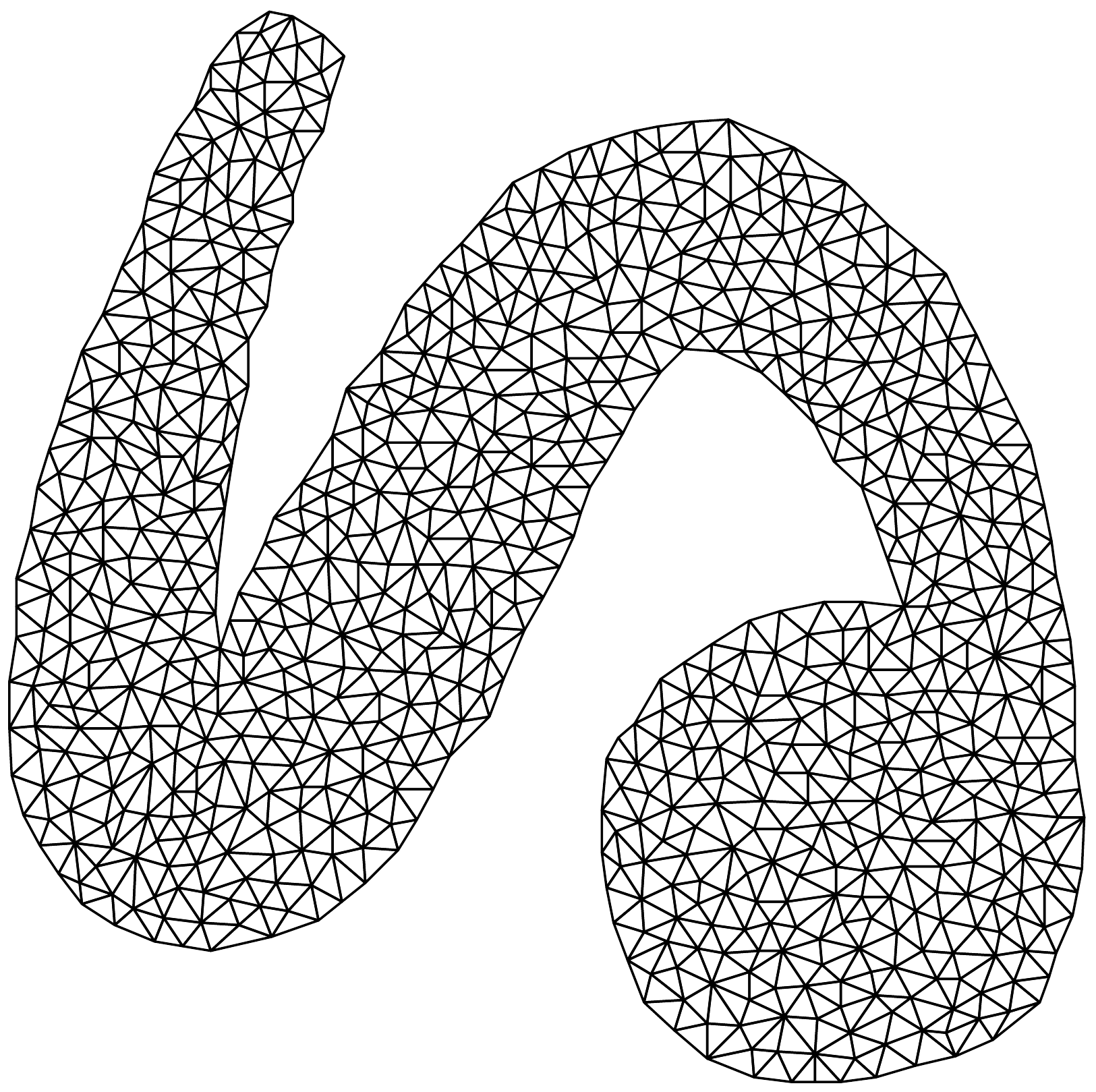}}\\
  \subfloat[dialog]

  {\includegraphics[width=0.45\linewidth]{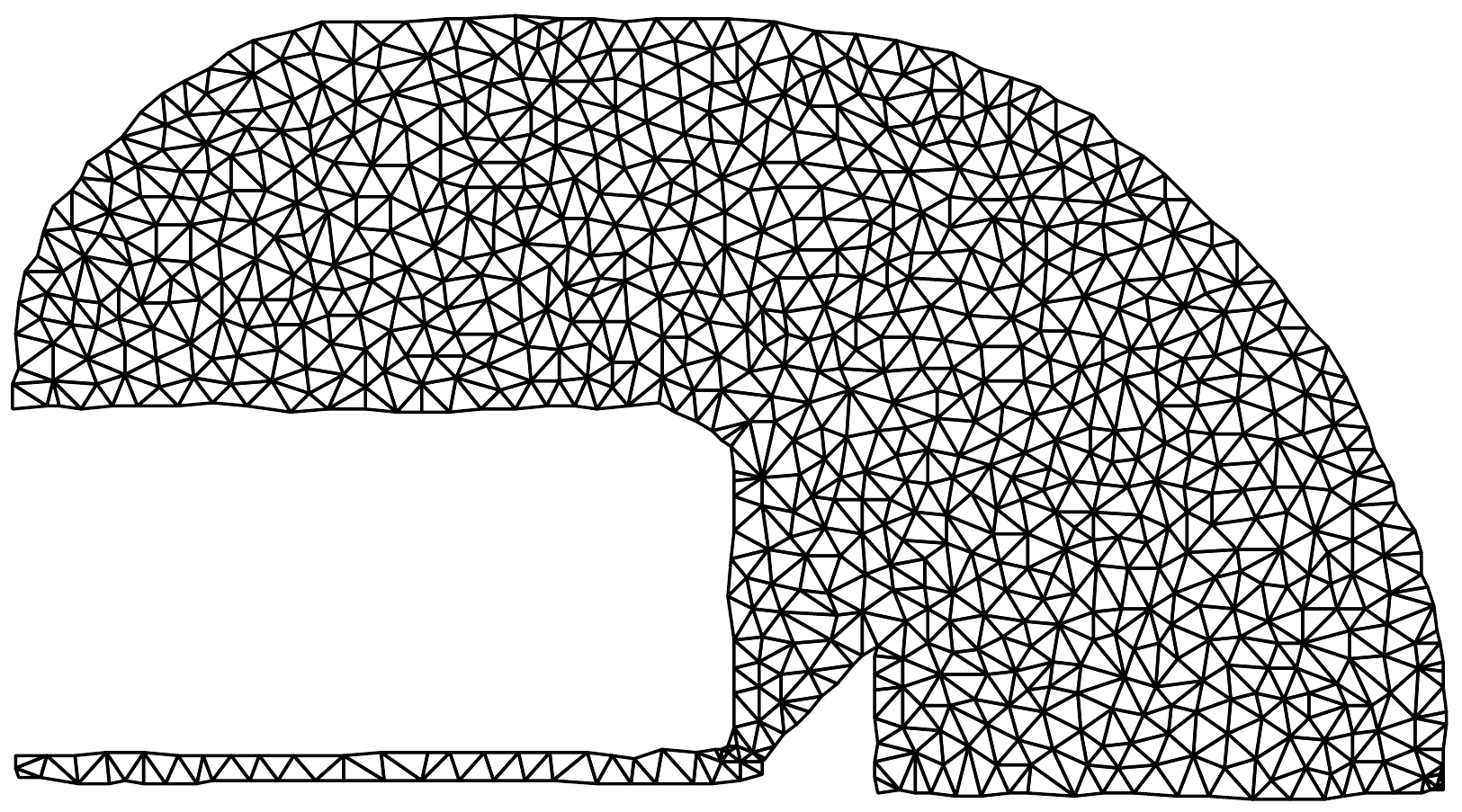}}~
  \subfloat[lake]
  {\includegraphics[width=0.45\linewidth]{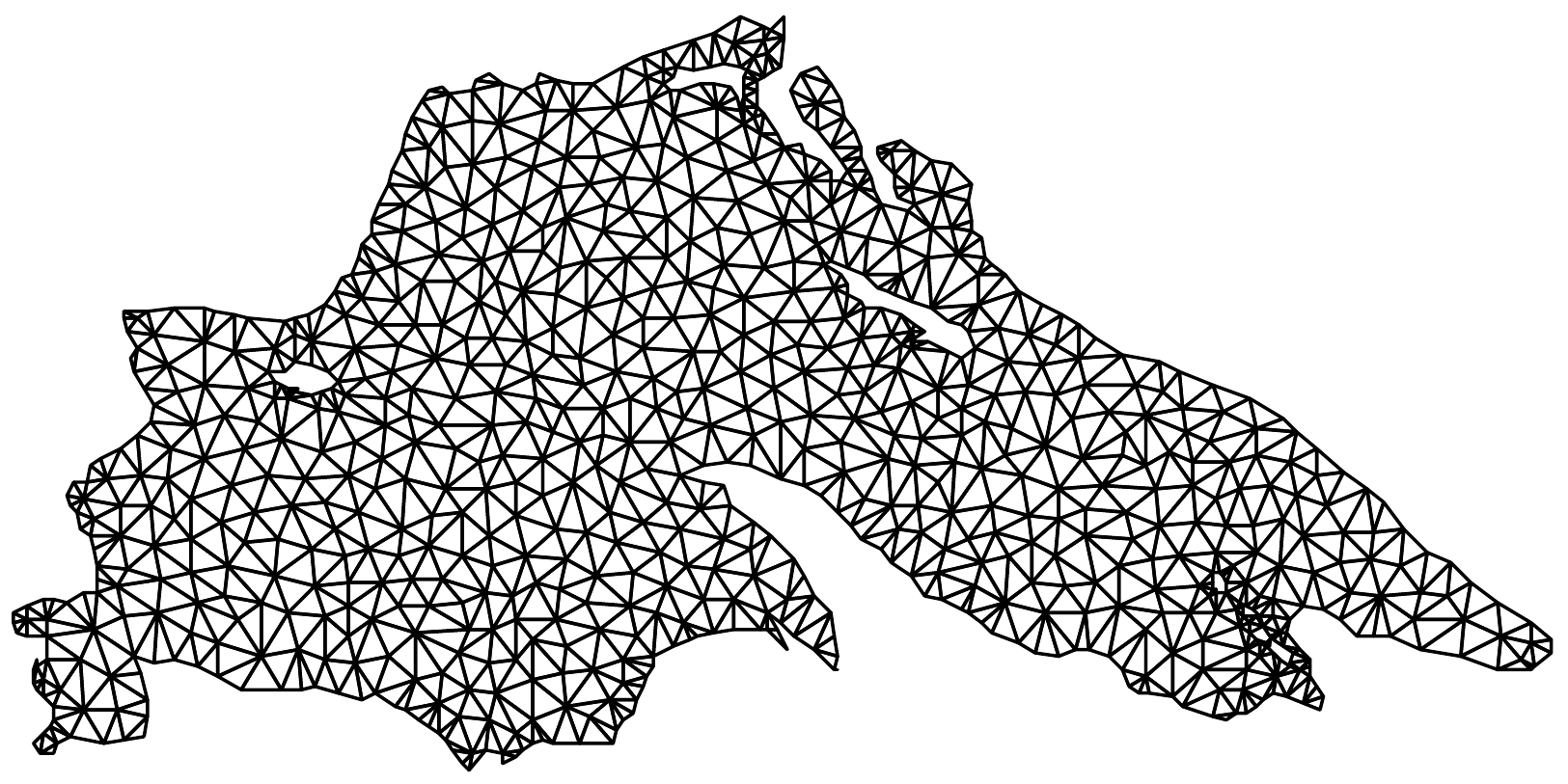}}\\
  \subfloat[riverflow]
  {\includegraphics[width=0.45\linewidth]{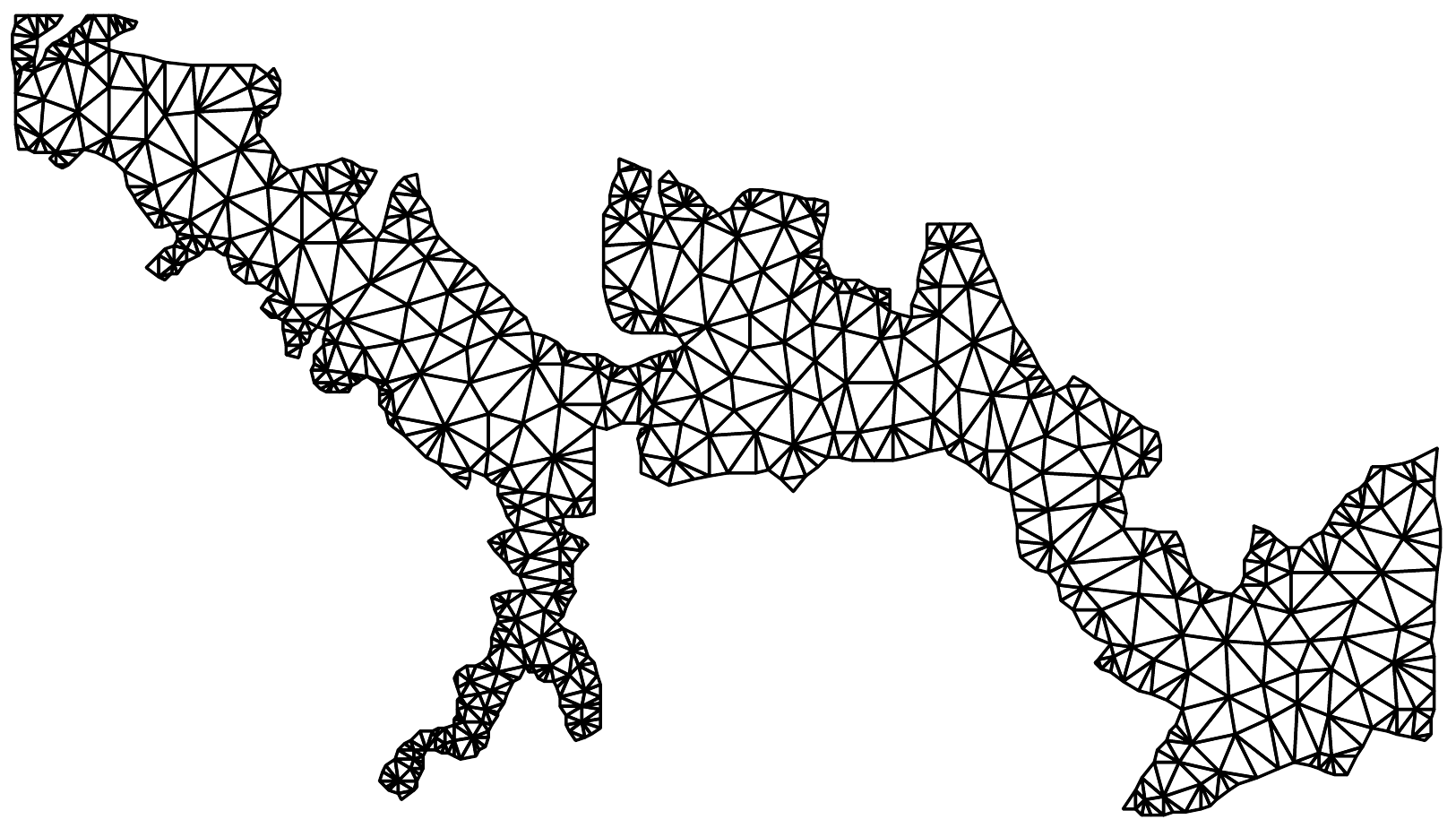}}~
  \subfloat[ocean]
  {\includegraphics[width=0.45\linewidth]{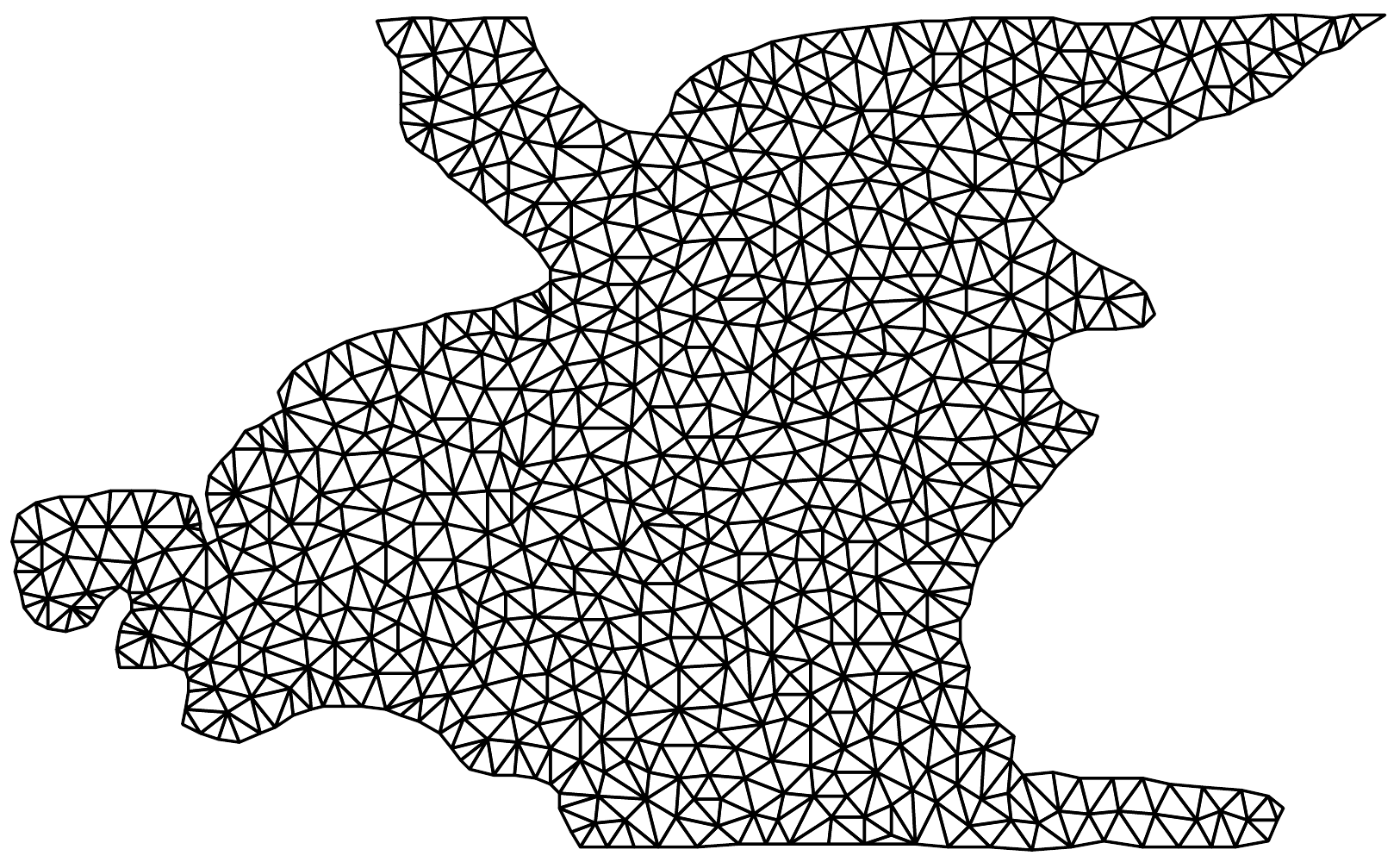}}\\
   \subfloat[stress]
  {\includegraphics[width=0.4\linewidth,angle=90]{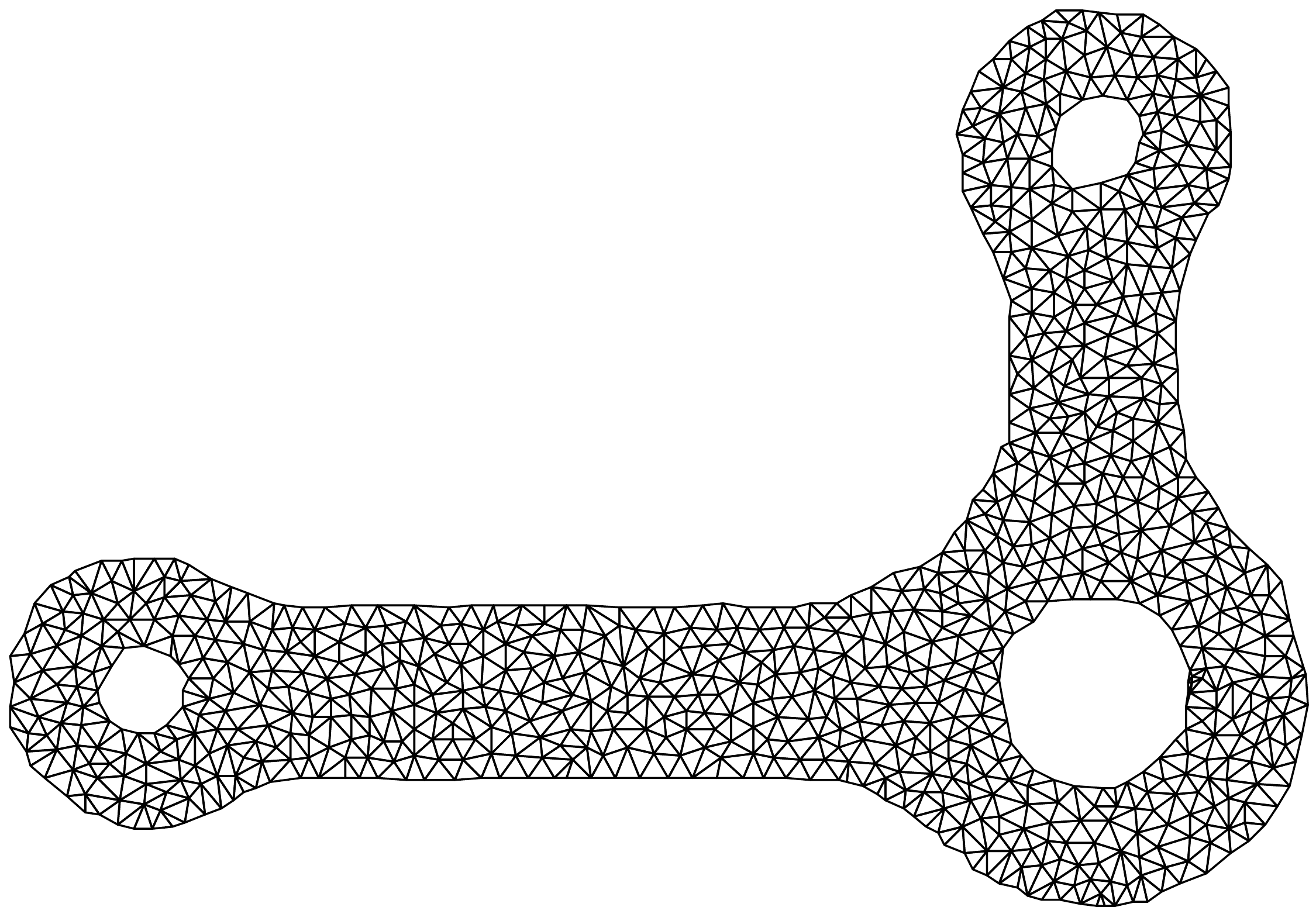}}~
   \subfloat[valve]
  {\includegraphics[width=0.45\linewidth]{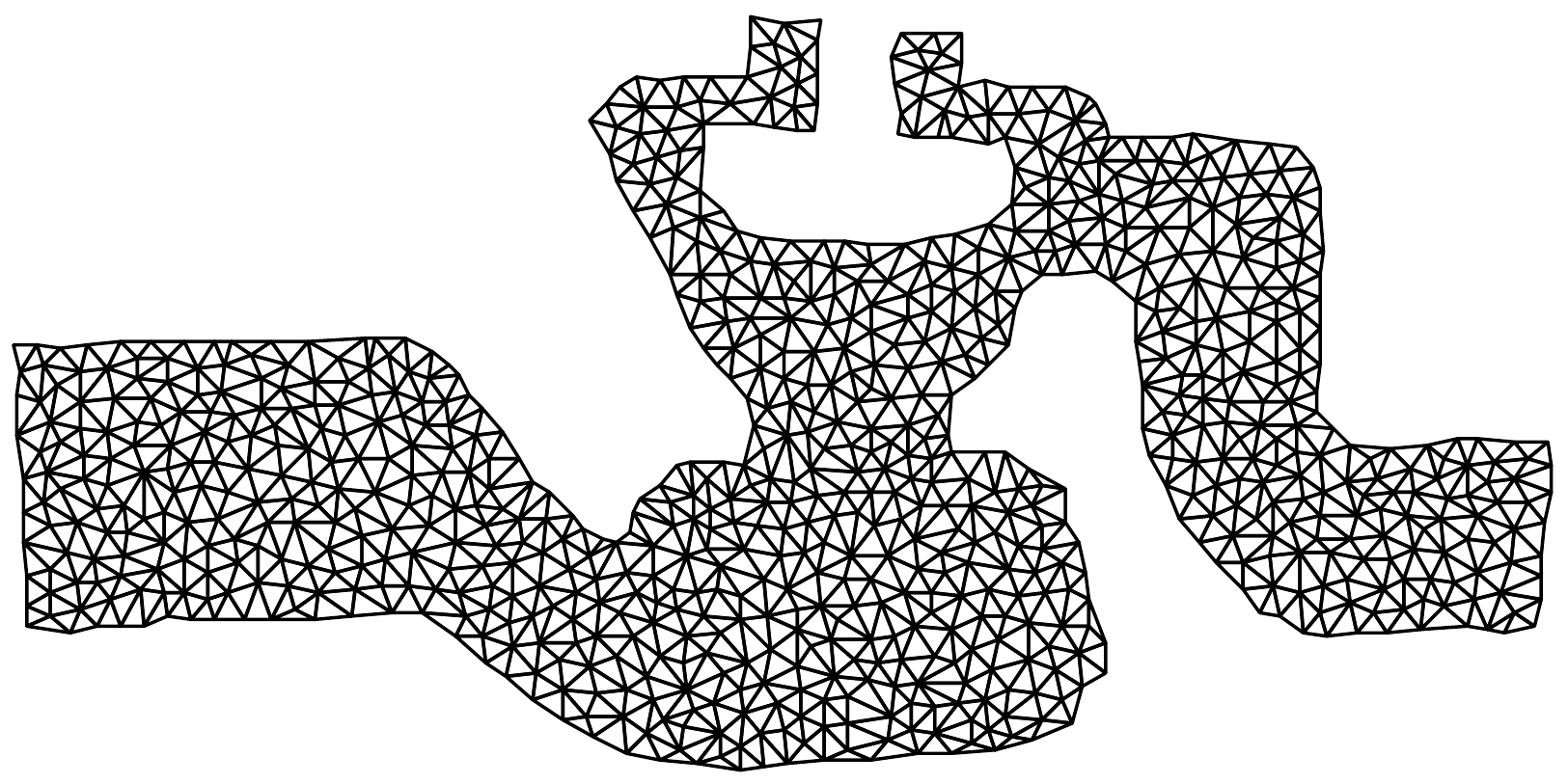}}\\
   \subfloat[wrench]
  {\includegraphics[width=0.45\linewidth]{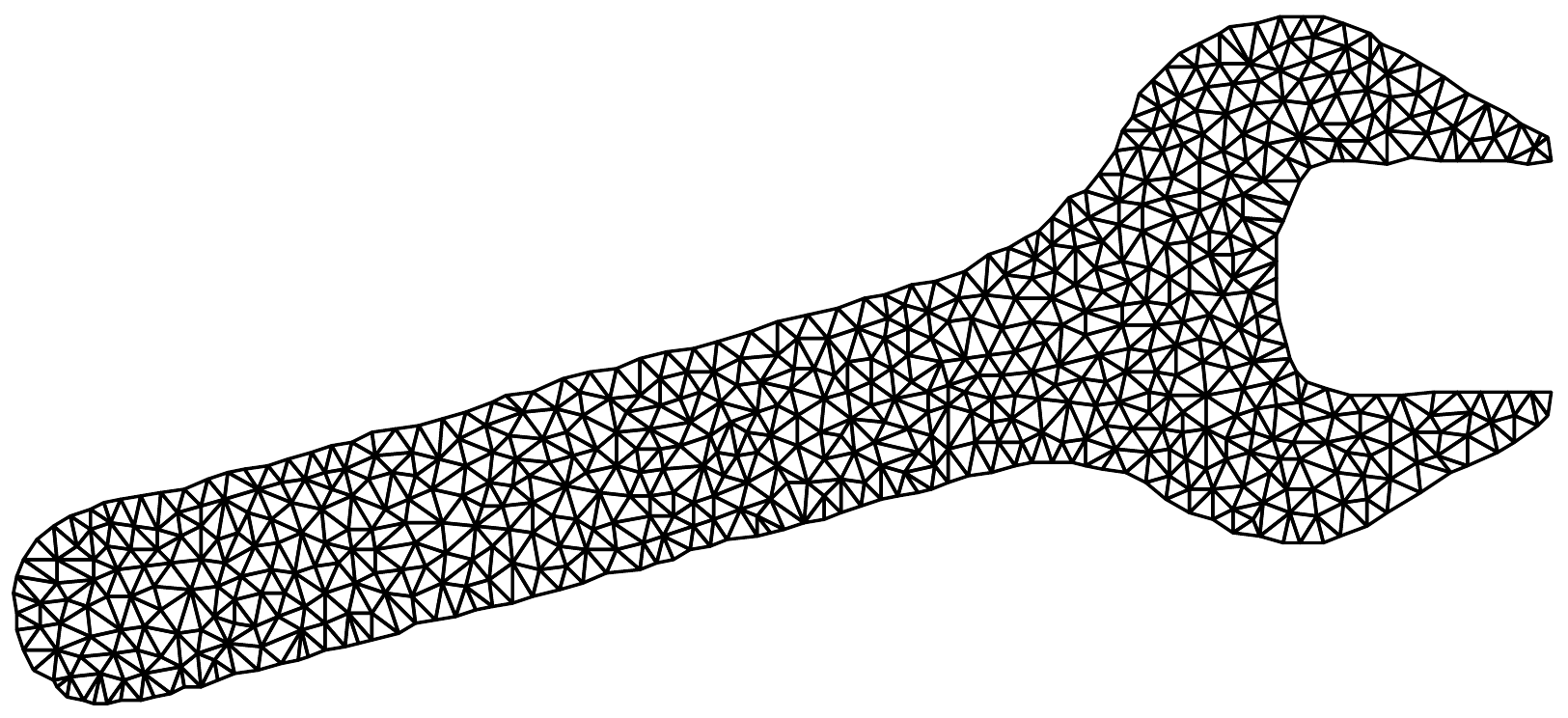}}
  \caption{2D meshes generated by Triangle~\cite{shewchuk96b} and used in the experiments (coarser but representative versions).
  \label{meshes}}
\end{figure}

\begin{table}
\small
\centering
\begin{tabular}{cccc}
\hline
     Label &       Mesh &     vertex &   triangle \\
\hline
        M1 &  carabiner &     328082 &     652920 \\
        M2 &      crake &     298898 &     595638 \\
        M3 &     dialog &     306824 &     611620 \\
        M4 &       lake &     375288 &     747676 \\
        M5 &  riverflow &     332699 &     661615 \\
        M6 &      ocean &     392674 &     783040 \\
        M7 &     stress &     312763 &     622868 \\
        M8 &      valve &     300985 &     599368 \\
        M9 &     wrench &     386757 &     771097 \\
\hline
\end{tabular}
\caption{Input Mesh Configuration}\label{meshinfo}
  \vspace{-0.5cm}
\end{table}

\subsection{Performance: Execution Time and Reduced Reuse Distance}

\subsubsection{Serial execution}
We first test our reuse distance reducing ordering on a serial run of Laplacian
mesh smoothing. Figure~\ref{exetime} shows the execution time results of
Laplacian mesh smoothing when \rdr was applied. For evaluation purpose, we 
provide the original (\ori) reordering (as computed by Triangle) and Breath
First Search (\bfs) reordering (developed by Strout et al.~\cite{MQ:Strout})
together. 
On the nine meshes, our algorithm got a 1.39 average speedup compared to \ori
and 1.19 speedup compared to \bfs.

\begin{figure}[!h]
\centering
\includegraphics[width=0.7\linewidth]{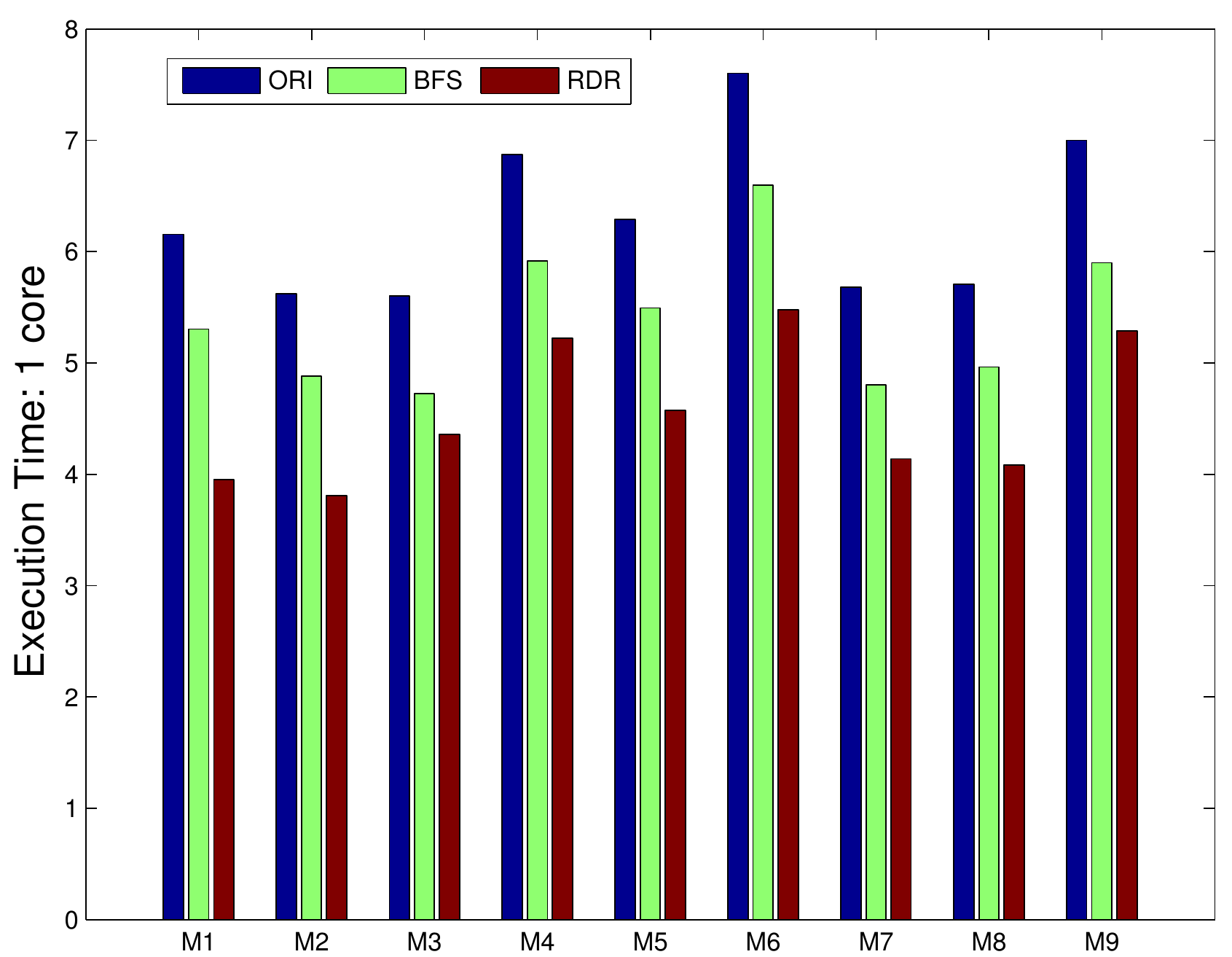}
\caption{Execution time (seconds) results for Laplacian mesh smoothing when
reuse distance reducing ordering was applied. \rdr ordering is 1.39 times faster
than \ori ordering.
\label{exetime}}
\end{figure}

\subsubsection{Cache Performance}
We collected the cache performance counter using PAPI 5.1.0.2~\cite{papi} to
compute cache performance.
Figures~\ref{fig:miss.l1}, \ref{fig:miss.l2} and~\ref{fig:miss.l3} show the L1,
L2, and L3 cache performance for Laplacian mesh smoothing running on a single
core for the different orderings. 
After the \rdr reordering was applied to Laplacian mesh smoothing, cache miss
rates significantly decreased. On average on one core, there were 25\% (resp.
6.3\%) fewer L1 cache misses, 71\% (resp. 51\%) fewer L2 cache misses and 84\%
(resp. 65\%) fewer L3 cache misses compared to \ori (resp. \bfs).

\begin{figure}[h!]
  \centering
  \subfloat[L1]
  {\includegraphics[width=0.7\linewidth]{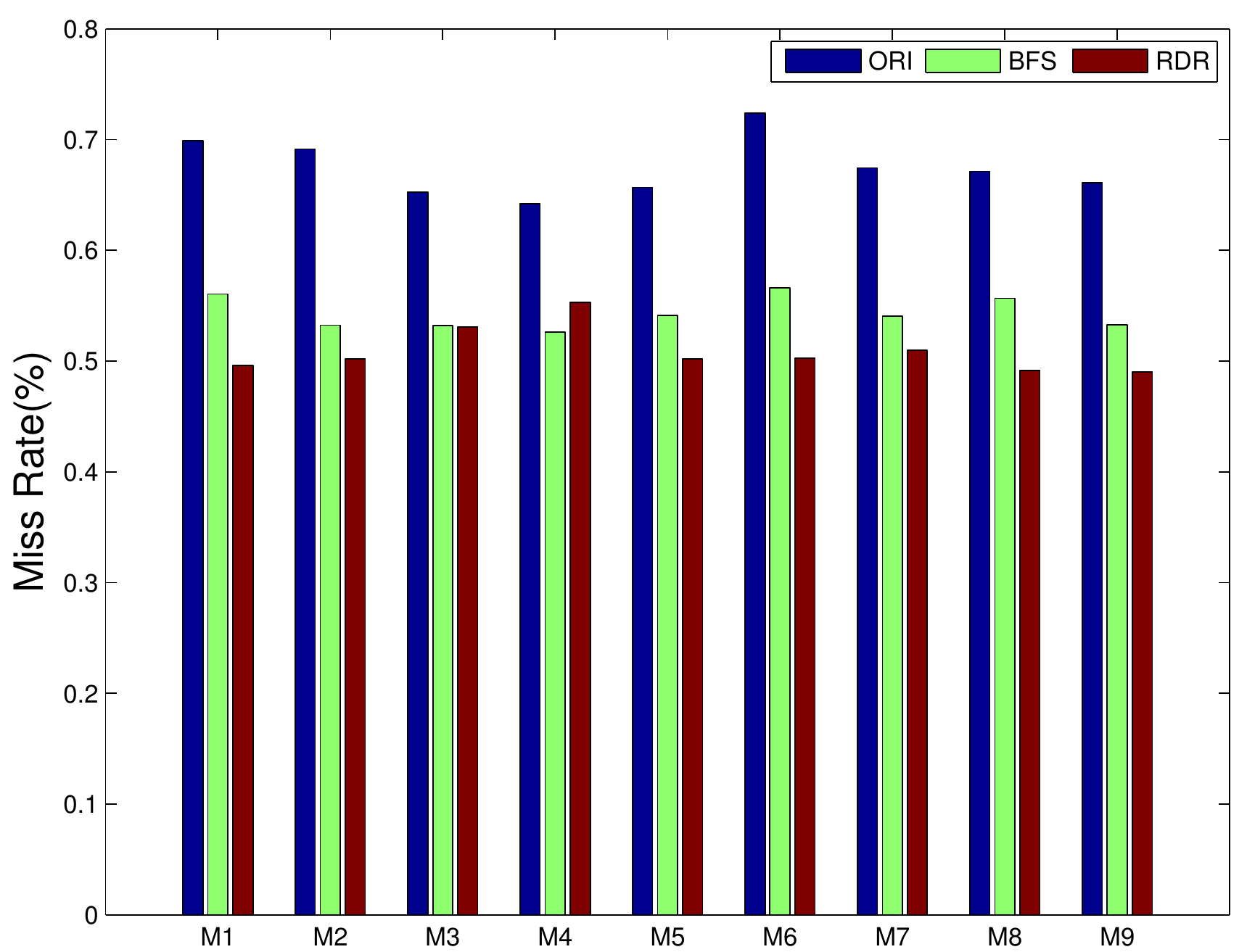} \label{fig:miss.l1}}\\
  \subfloat[L2]
  {\includegraphics[width=0.7\linewidth]{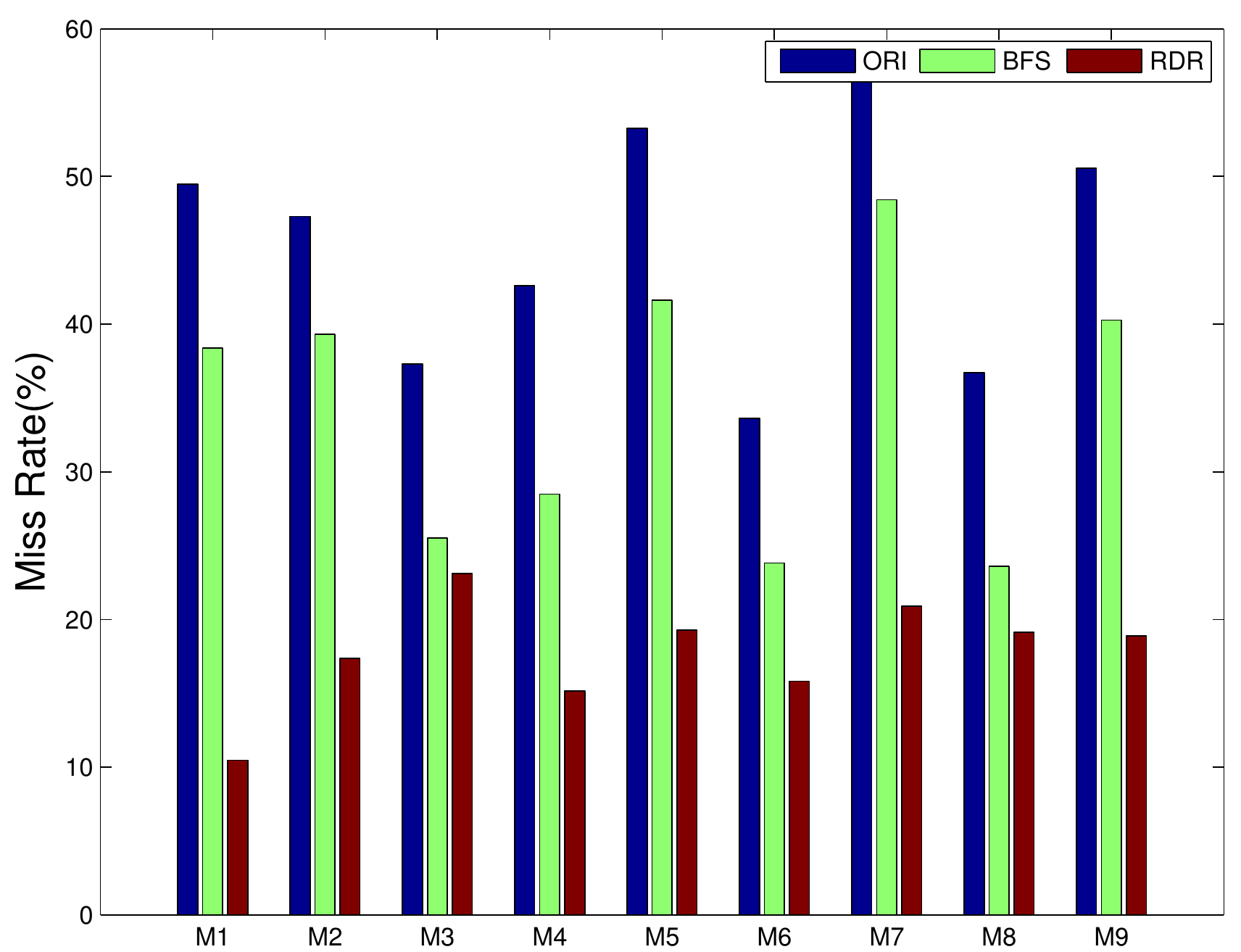} \label{fig:miss.l2}}\\
  \subfloat[L3]
  {\includegraphics[width=0.7\linewidth]{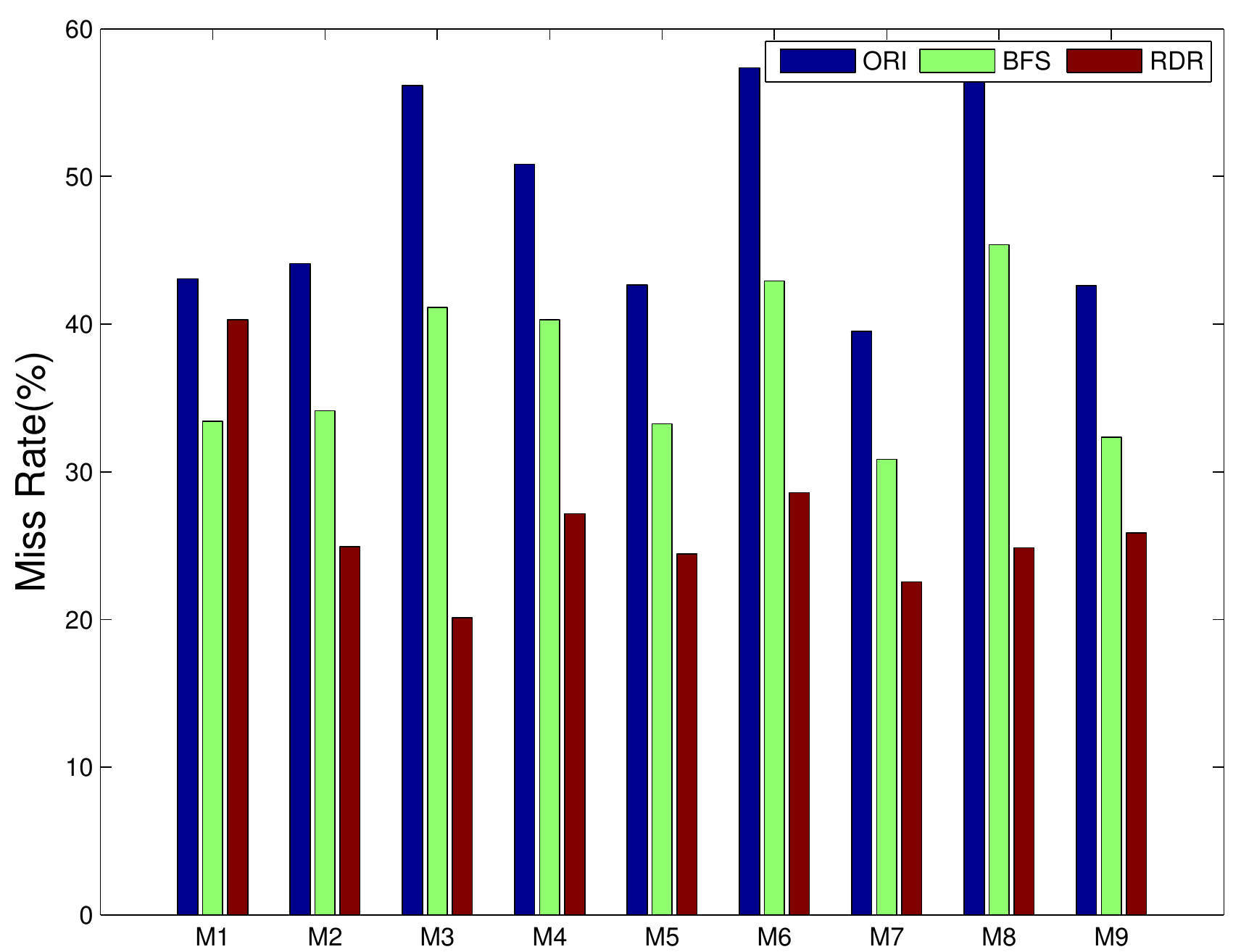} \label{fig:miss.l3}}
\caption{Cache performance results on one core when reordering were applied. The
LX miss rate depends on the number of LX accesses: a higher L2 (resp. L3) miss
rate does not necessarily means higher number of misses, they need to be put
in relation with the number of L1 (resp. L1 and L2) misses.
\label{fig:cachemiss}}
    \vspace{-0.5cm}
\end{figure}

To understand how this affects the execution time, let us call $m_1$ (resp.
$m_2$, $m_3$) the cache miss rates of L1 (resp. L2, L3), and $c_2$ (resp.
$c_3$, $c_m$) the cost of a cache access to L2 (resp. L3, Memory). Then let
\#accesses be the number of data accesses, the additional cost to the execution
time is:
\begin{equation}
(m_1\cdot c_2 + m_1 m_2\cdot c_3 + m_1m_2m_3\cdot c_m)\cdot \text{\#accesses}.
\end{equation}
Indeed, $m_1\cdot \text{\#accesses}$ operations are not found in L1 and are
fetched in L2 (hence an additional cost of $c_2$), out of those, $m_2\cdot
\text{\#accesses}$ are not found in L2 and have to be fetched in L3 etc.

To give an example of what it represents, for Carabiner, with the access costs
of the machine (see Section~\ref{sec:setup}), \texttt{ori} has 927k additional
clock cycles due to cache misses, \texttt{bfs} has 528k, and \texttt{rdr} has
210k.

Similar cache performance is obtained when the Laplacian mesh smoothing runs on
multicores (32 cores).

\subsubsection{Reuse Distance study}
To understand the speedups observed during the serial execution, we study the
reuse distance of the different orderings. To measure the reuse distance,
because this functionality is not offered by PAPI, we did a verbose run noting
the data locations being addressed and analyzed them.

We present some of the quantiles observed in Table~\ref{tab:quantiles}. The $X$
quantile is defined as the smallest value such that there at least a proportion
$X$ of the population below it.
\begin{table*}[h!]
\centering
\begin{tabular}{|l|c|rrrr|c|}
\cline{3-7}
\multicolumn{2}{c|}{}&\multicolumn{4}{c|}{Quantiles}& \#accesses\\
\hline
mesh & {\centering Ordering} & {\centering $50\%$} &  {\centering $75\%$} &  {\centering $90\%$} &  {\centering $100\%$} &\\
\hline
	&	\ori	&	8	&	52	&	1,168	&	1,924,021	&		\\
carabiner	&	\bfs	&	1	&	11	&	99	&	1,923,989	&	15,566,520	\\
	&	\rdr	&	1	&	4	&	6	&	1,942	&		\\
\hline													
	&	\ori	&	8	&	43	&	642	&	1,767,468	&		\\
crake	&	\bfs	&	1	&	11	&	80	&	1,767,488	&	14,226,264	\\
	&	\rdr	&	1	&	4	&	6	&	3,903	&		\\
\hline													
	&	\ori	&	7	&	39	&	306	&	1,819,234	&		\\
dialog	&	\bfs	&	1	&	10	&	79	&	1,803,850	&	14,614,336	\\
	&	\rdr	&	1	&	5	&	11	&	6,198	&		\\
\hline													
	&	\ori	&	7	&	38	&	270	&	2,224,176	&		\\
lake	&	\bfs	&	1	&	10	&	72	&	2,224,069	&	17,850,952	\\
	&	\rdr	&	1	&	4	&	5	&	8,774	&		\\
\hline													
	&	\ori	&	7	&	37	&	1,430	&	1,963,783	&		\\
riverflow	&	\bfs	&	1	&	10	&	69	&	1,969,058	&	15,758,200	\\
	&	\rdr	&	1	&	4	&	6	&	5,429	&		\\
\hline													
	&	\ori	&	7	&	38	&	3,866	&	2,312,947	&		\\
ocean	&	\bfs	&	1	&	10	&	76	&	2,339,305	&	18,719,512	\\
	&	\rdr	&	1	&	5	&	9	&	6,122	&		\\
\hline													
	&	\ori	&	7	&	44	&	358	&	1,817,752	&		\\
stress	&	\bfs	&	1	&	10	&	91	&	1,817,758	&	14,862,888	\\
	&	\rdr	&	1	&	4	&	6	&	4,809	&		\\
\hline													
	&	\ori	&	7	&	40	&	646	&	1,780,959	&		\\
valve	&	\bfs	&	1	&	10	&	75	&	1,780,916	&	14,301,320	\\
	&	\rdr	&	1	&	4	&	6	&	7,256	&		\\
\hline													
	&	\ori	&	7	&	42	&	376	&	2,290,783	&		\\
wrench	&	\bfs	&	1	&	10	&	89	&	2,290,734	&	18,429,528	\\
	&	\rdr	&	1	&	4	&	6	&	6,487	&		\\
\hline
\end{tabular}
\caption{Distribution of the reuse distances of the first iterations of each run
as a function of the meshes and the orderings.\label{tab:quantiles}}
\end{table*}

What is notable is that for most of the data-accesses, the reuse distance is
well below that of the L1 cache: as an order of magnitude, assuming that each
node is 66 bytes\footnote{A node is characterized by (i) its coordinates: two
double precision floats, (ii) its connectivity: an array of long integers
(typically here five or six neighbors) and (iii) a fixed/boundary state: an
integer that says whether or not the vertex is an interior one~\cite{mesquite}.
Note that this is an approximation as other factors have to be accounted for
(for instance the data-structure used to store the node information), in
practice the size can be many more times this.},
and that we use the theoretical cache miss model (Section~\ref{sec:reuse}); then
below a reuse distance of 496 (resp. 3970; 372,000) there should not be any L1
(resp. L2; L3) cache miss. This is consistent with the findings that the
L1-cache miss rate is very low for all orderings (see
Figure~\ref{fig:cachemiss}).

It is however important to notice that the maximum reuse distance of \rdr is
well below the theoretical reuse distance for a L3 cache miss (at least 42 times
smaller). Intuitively, even if our approximation is only a first order
approximation, this means that none of the L3 cache misses observed by PAPI in
\rdr are due to elements being reused. Most certainly they are either compulsory
cache misses or due to the first fetching of a given element.

Based on the measured cache miss rates, we have been able to estimate the actual
number of misses which we represent in Table~\ref{tab:number_miss} (Estim.
number of misses).
We have seen in Table~\ref{tab:quantiles} that in practice the
L3 cache misses observed in \rdr were not results of the reuse distance but of
other factors, hence we could compute them for all graphs\footnote{In practice, when
computing this number, we were able to verify that it was almost constant for
all graphs (between 400 and 500), hence confirming the intuition that it was due
to external factors.} and subtract them from the total number of L1, L2 and
L3 misses for all orderings.\footnote{Results were similar when we did not
perform this operation.} 
The results obtained allow us to go further in the understanding of the behavior
of the different algorithms.

\begin{table}[!h]
\centering
\resizebox{\linewidth}{!}{
\begin{tabular}{|l|c|rrr|rrr|}
\cline{3-8}
\multicolumn{2}{c|}{}&\multicolumn{3}{c|}{Estim. number}&\multicolumn{3}{c|}{Estim. max number}\\
\multicolumn{2}{c|}{}&\multicolumn{3}{c|}{of misses (x$10^3$)}&\multicolumn{3}{c|}{of elements (x$10^3$)}\\
\hline
mesh & {\centering Ordering} & {\centering L1} &  {\centering L2} &  {\centering L3}& {\centering L1} &  {\centering L2} &  {\centering L3} \\
\hline
	&	\ori	&	106	&	50.6	&	19.9	&	13.2	&	21.3	&	330	\\
carabiner	&	\bfs	&	84	&	30.2	&	7.94	&	10.2	&	21.2	&	1060	\\
	&	\rdr	&	73.9	&	4.82	&	0	&	1.6	&	1.88	&	1.94	\\
\hline															
	&	\ori	&	95.3	&	43.4	&	17.4	&	24.6	&	40.9	&	198	\\
crake	&	\bfs	&	72.6	&	26.7	&	7.07	&	18.3	&	39.2	&	986	\\
	&	\rdr	&	68.3	&	9.3	&	0	&	3.4	&	3.77	&	3.9	\\
\hline															
	&	\ori	&	91.7	&	31.9	&	16.3	&	59	&	87.7	&	108	\\
dialog	&	\bfs	&	74.1	&	16.2	&	4.55	&	53.2	&	89.3	&	157	\\
	&	\rdr	&	74	&	14.3	&	0	&	5.84	&	6.05	&	6.2	\\
\hline															
	&	\ori	&	111	&	44.7	&	20.7	&	62.2	&	95.9	&	143	\\
lake	&	\bfs	&	89.8	&	22.7	&	6.72	&	56.1	&	96.1	&	219	\\
	&	\rdr	&	94.7	&	10.9	&	0	&	8.42	&	8.65	&	8.77	\\
\hline															
	&	\ori	&	99.8	&	51.4	&	19.8	&	30.4	&	66.8	&	863	\\
riverflow	&	\bfs	&	81.6	&	31.8	&	8.08	&	20.3	&	67.6	&	1290	\\
	&	\rdr	&	75.4	&	11.5	&	0	&	4.73	&	5.27	&	5.43	\\
\hline															
	&	\ori	&	131	&	41.3	&	21.9	&	63.3	&	300	&	1050	\\
ocean	&	\bfs	&	102	&	21	&	6.57	&	41.6	&	770	&	1830	\\
	&	\rdr	&	89.8	&	10.6	&	0	&	5.83	&	5.99	&	6.12	\\
\hline															
	&	\ori	&	96.6	&	54.6	&	19.4	&	71.5	&	91.7	&	130	\\
stress	&	\bfs	&	76.7	&	35.3	&	8.42	&	59.1	&	87.6	&	137	\\
	&	\rdr	&	72.2	&	12.3	&	0	&	4.36	&	4.66	&	4.81	\\
\hline															
	&	\ori	&	92.6	&	31.9	&	16.9	&	37.3	&	77.6	&	186	\\
valve	&	\bfs	&	76.3	&	15.4	&	5.17	&	28.6	&	110	&	1040	\\
	&	\rdr	&	67	&	10.1	&	0	&	6.68	&	7.17	&	7.26	\\
\hline															
	&	\ori	&	117	&	57.1	&	21.8	&	75.3	&	105	&	148	\\
wrench	&	\bfs	&	93.8	&	35.1	&	8.38	&	62.9	&	103	&	170	\\
	&	\rdr	&	86	&	12.7	&	0	&	5.69	&	6.31	&	6.49	\\
\hline
\end{tabular}
}
\caption{Estimated number of misses for the different levels of cache. Based on
those, we estimated the maximum number of elements (using reuse distance)
that could fit in the caches.\label{tab:number_miss}}
\end{table}

In order to do this, we again used the theoretical model introduced in
Section~\ref{sec:reuse} to measure the maximum number of elements that fit in
the cache. We formulate this model:\\ 
\emph{Assuming that there are $n_X$ LX misses, then the $n_X$ accesses with the
largest reuse distance are the one that missed}.

With this approximation we measured the number of elements that could fit in the
cache for each graph and each ordering. We report those values in
Table~\ref{tab:number_miss} (section Estim. max number of elements).
The first observation is that in general, given a graph, for both orderings \ori
and \bfs, the estimated maximum number of elements that fit the different cache
are similar (same order of magnitude). This makes sense as for a given graph the
elements have the same size\footnote{Note that this is not as true for
\texttt{ocean} and \texttt{valve}. We have no explanation for these graphs.}
--not necessarily 66 bytes which was an approximation--.

This is particularly visible for the L2 cache. L1 is smaller, hence it is more
sensible to the cache optimizations and hence to the first order approximation. 
On the contrary, the number of accesses in L3 are smaller, hence they are more
sensible to the noise incurred by other elements of the computation that are not
reported by PAPI. The number of accesses of L2 and its size are the right
balance to verify our approximation.

The second observation, is that the estimated number of elements that fit the
L2 cache for \rdr is many orders of magnitude below that of the other
orderings. On the contrary it is very close to that of the L1 cache miss of
\rdr. Intuitively, this tends to give the idea that there are actually no L2
cache miss with the ordering, and that the L2 cache misses are, similarly to the
L3 cache misses linked to other factors (for this study we only take into
account the reuse distance).

This would show that the \rdr ordering not only does not have any L3 cache
misses, but also have very few L2 cache misses, making it quasi-optimal.

\subsection{Performance Scalability}
Finally, we have studied the scalability of the reordering. In order to do so
we have run the Laplacian mesh smoothing algorithm on 1, 2, 4, 8, 16, 24 and 32
cores.

In Figure~\ref{fig:speedup}, we compare the different speedups per mesh. Each
speedup is relative to a serial baseline (the execution time of \ori on one
core), hence given by the formula:
\[
\text{Speedup}(\texttt{ordering},p) = \frac{T_{\ori}(1)}{T_{\texttt{ordering}}(p)}
\]
\begin{figure}[!h]
  \centering
  \subfloat[1 core]
  {\includegraphics[width=0.48\linewidth]{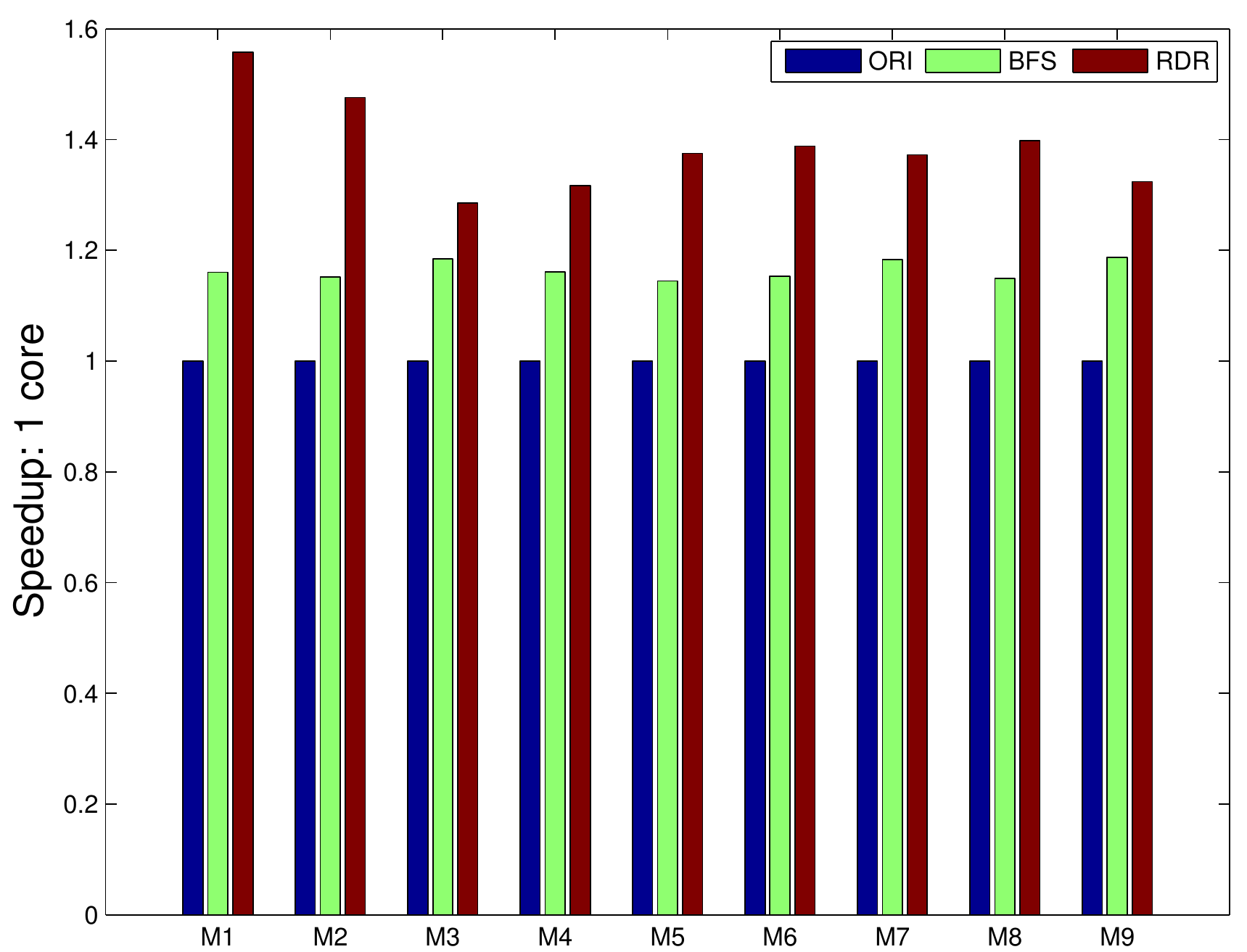}}
  \subfloat[2 cores]
  {\includegraphics[width=0.48\linewidth]{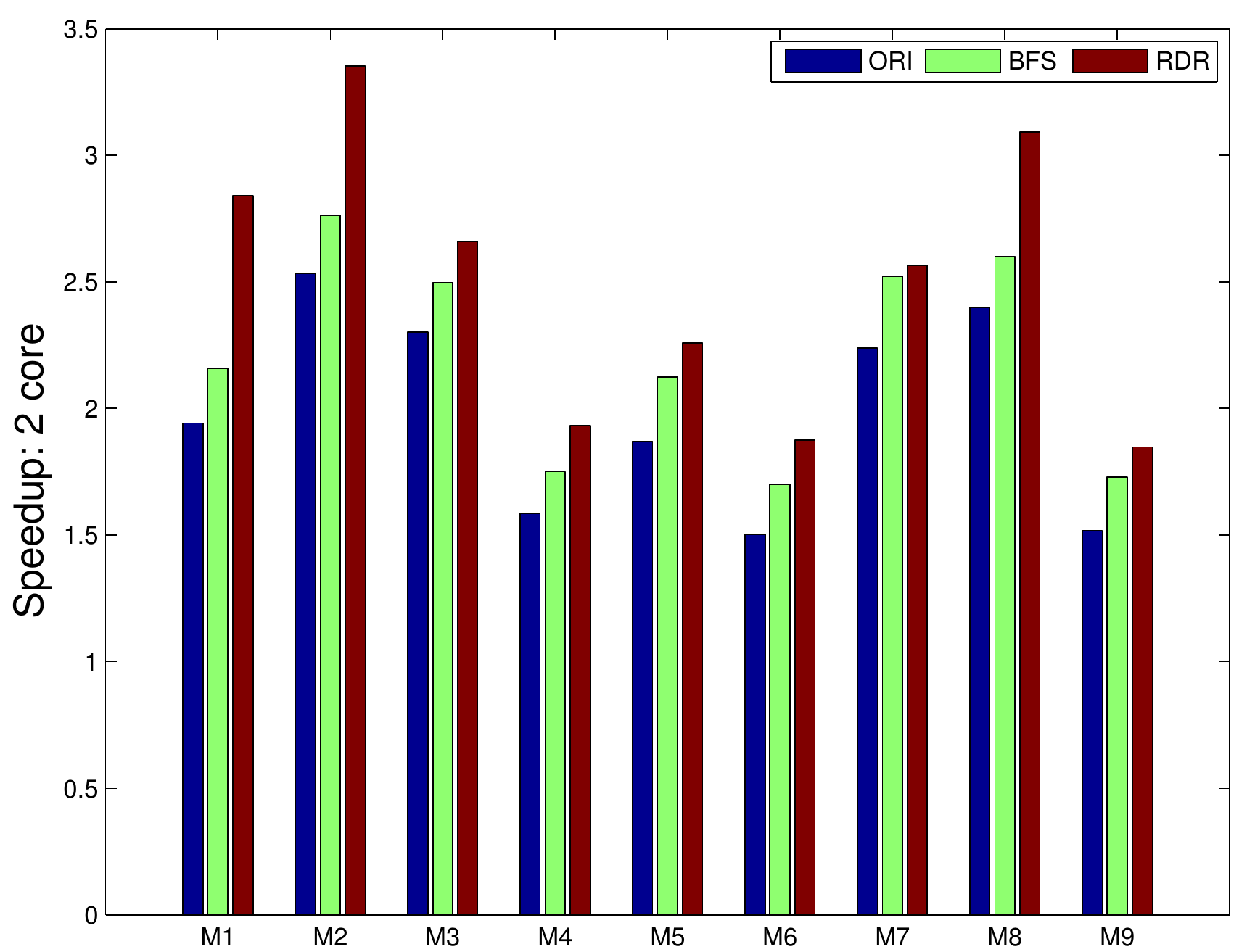}} \\
  \subfloat[4 cores]
  {\includegraphics[width=0.48\linewidth]{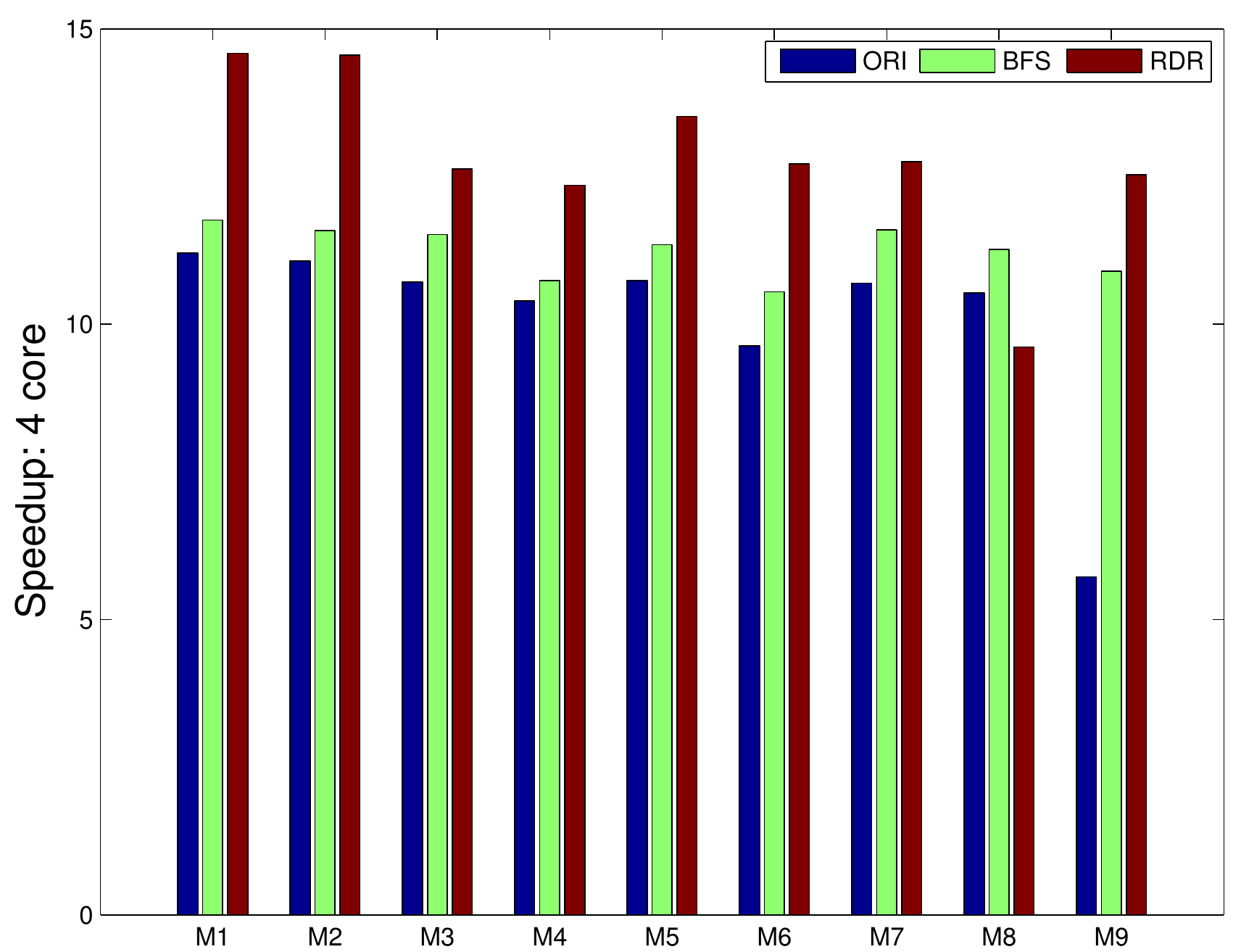}}
  \subfloat[8 core]
  {\includegraphics[width=0.48\linewidth]{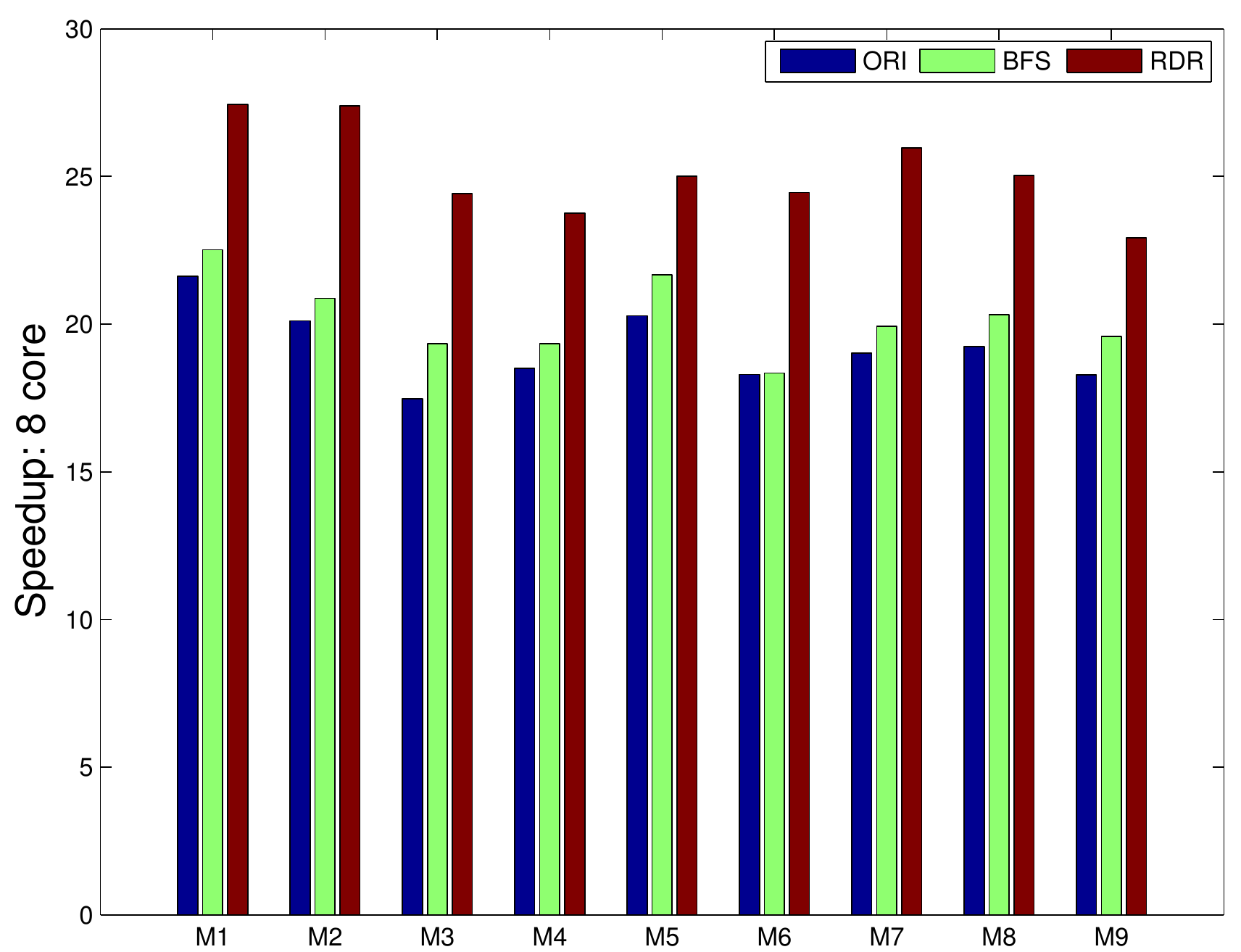}}\\
  \subfloat[16 cores]
  {\includegraphics[width=0.48\linewidth]{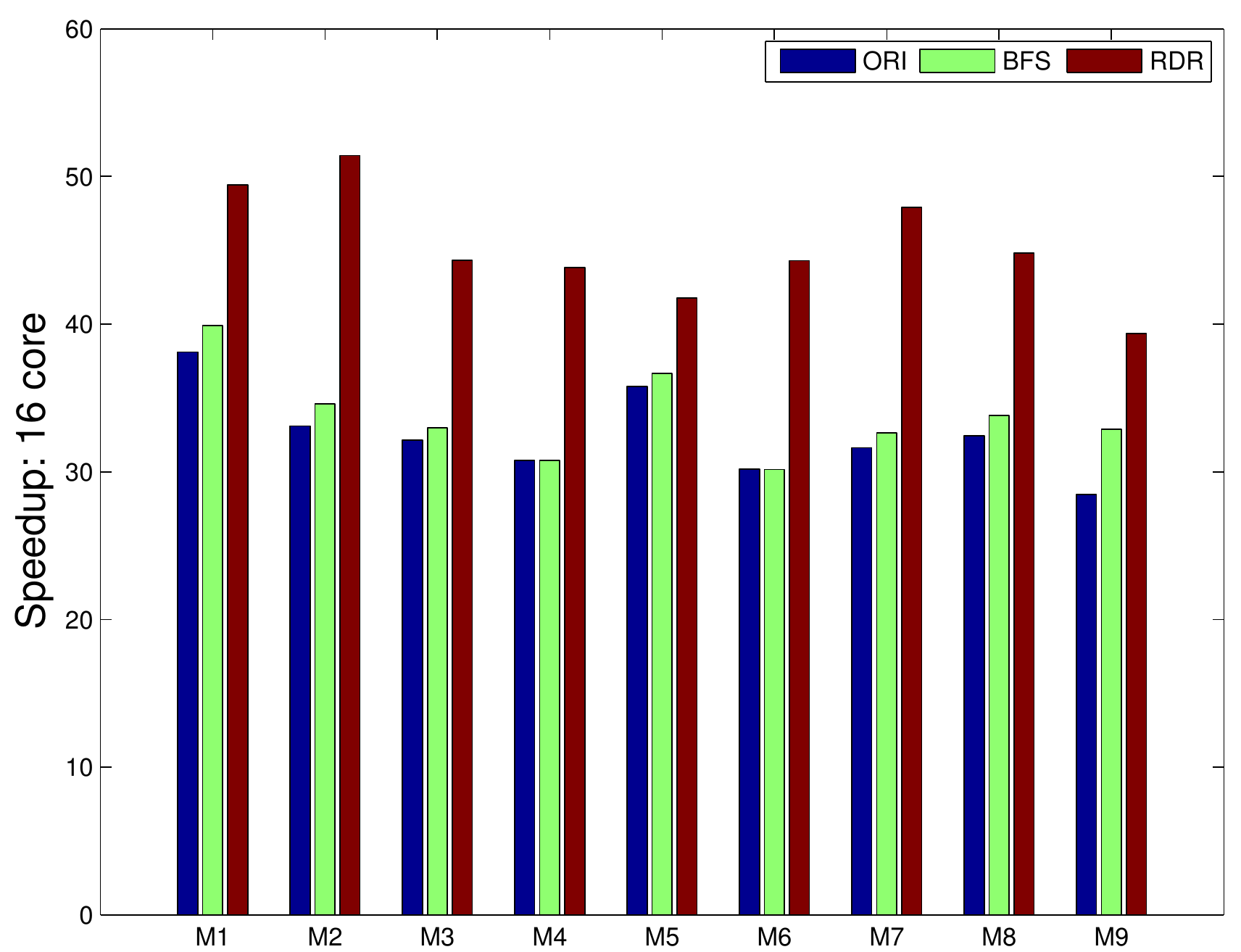}}
  \subfloat[24 cores]
  {\includegraphics[width=0.48\linewidth]{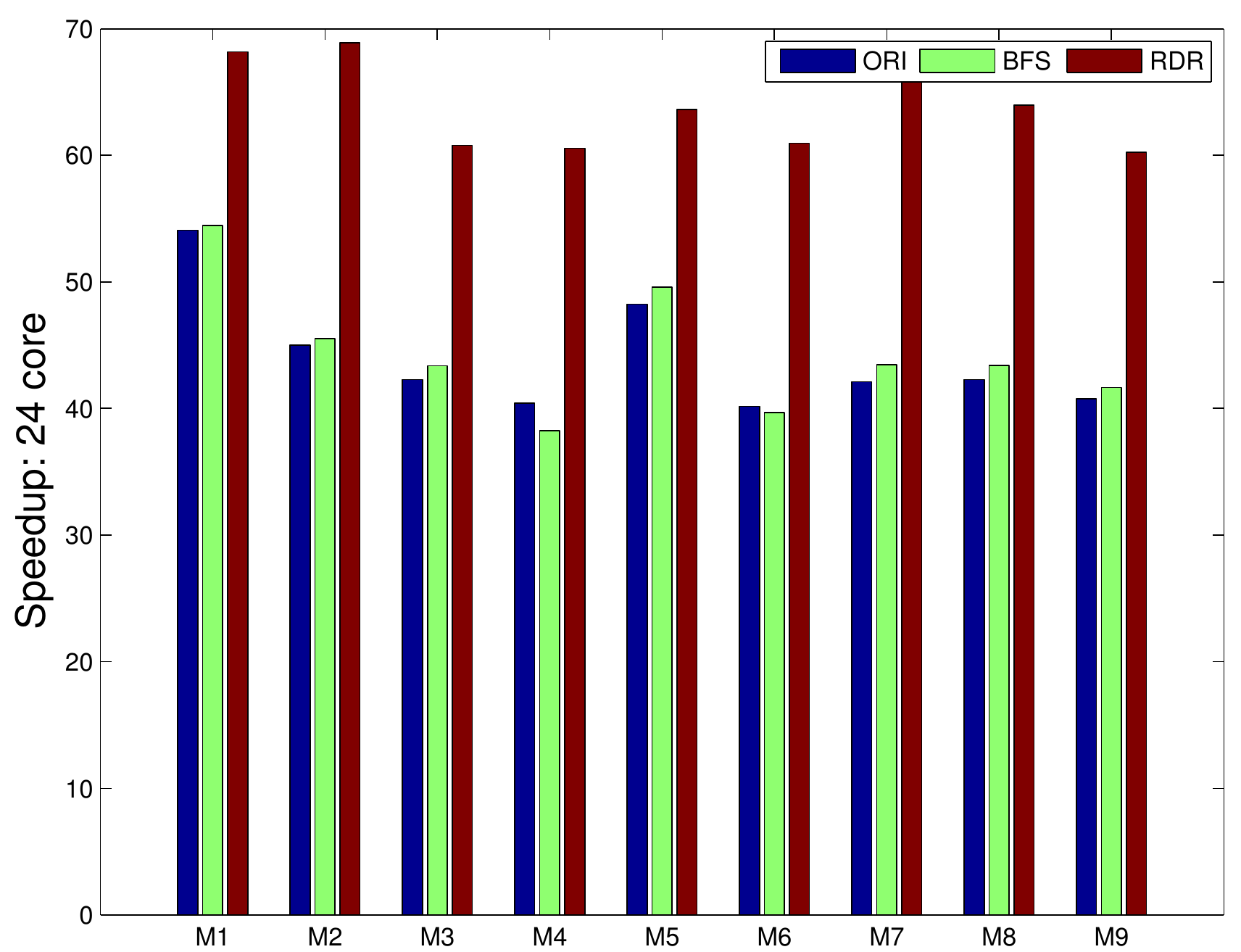}}\\
  \subfloat[32 cores]
  {\includegraphics[width=0.48\linewidth]{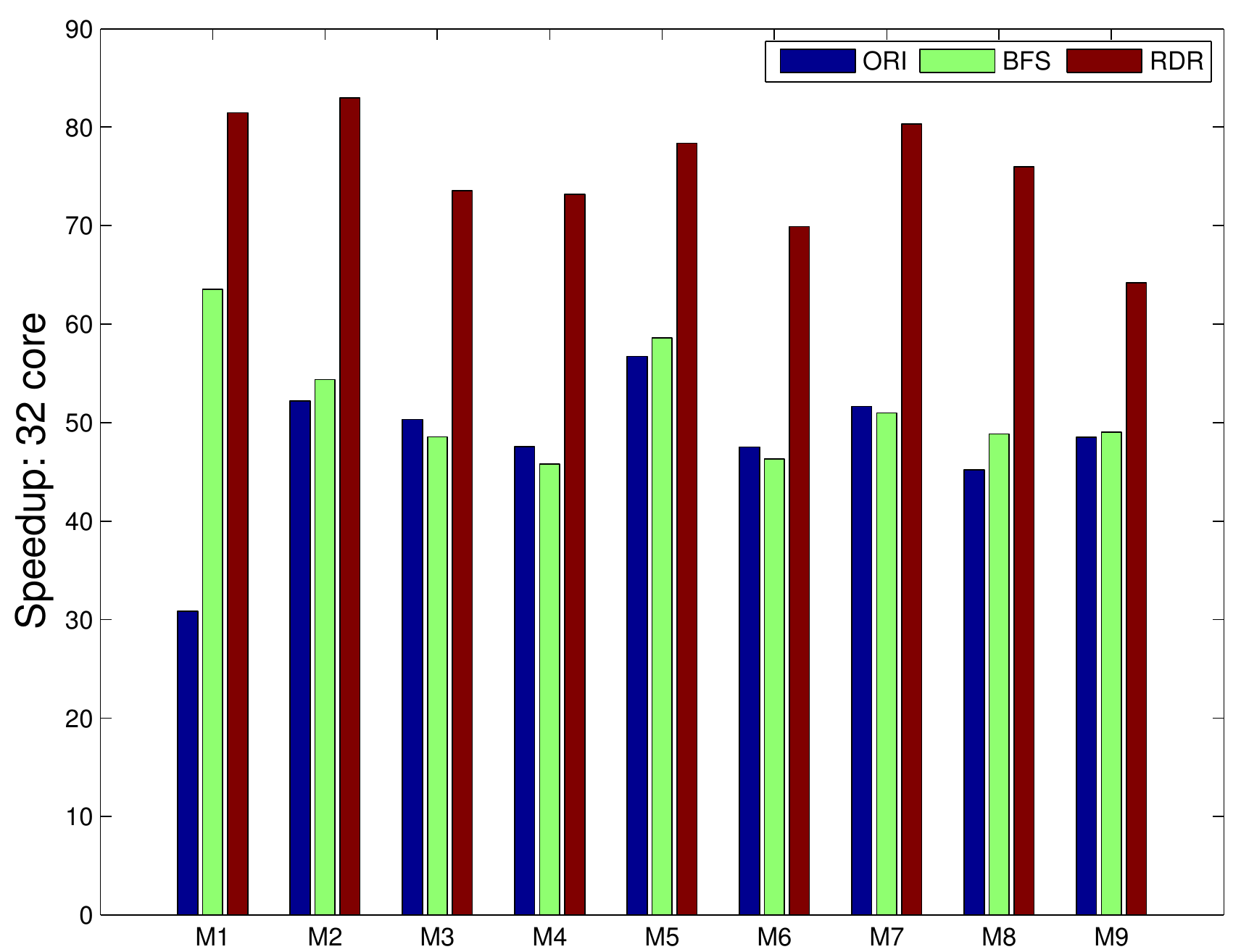}}
  \caption{Observed speedup relative to the serial \ori baseline. Reuse distance reducing ordering (\rdr) provides significant performance improvement. }
  \label{fig:speedup}
  \vspace{-0.5cm}
\end{figure}

The first notable result is that even with only the original ordering (\ori) the
speedup is supra-linear. One way to explain this result would be a better data
management when using multiple cores. We plot the number of additional accesses
for the three first meshes Carabiner, Crake and Dialog in
Figure~\ref{fig:scal_accesses}. Results are similar for the remaining meshes.
\begin{figure}[!h]
\includegraphics[width=\linewidth]{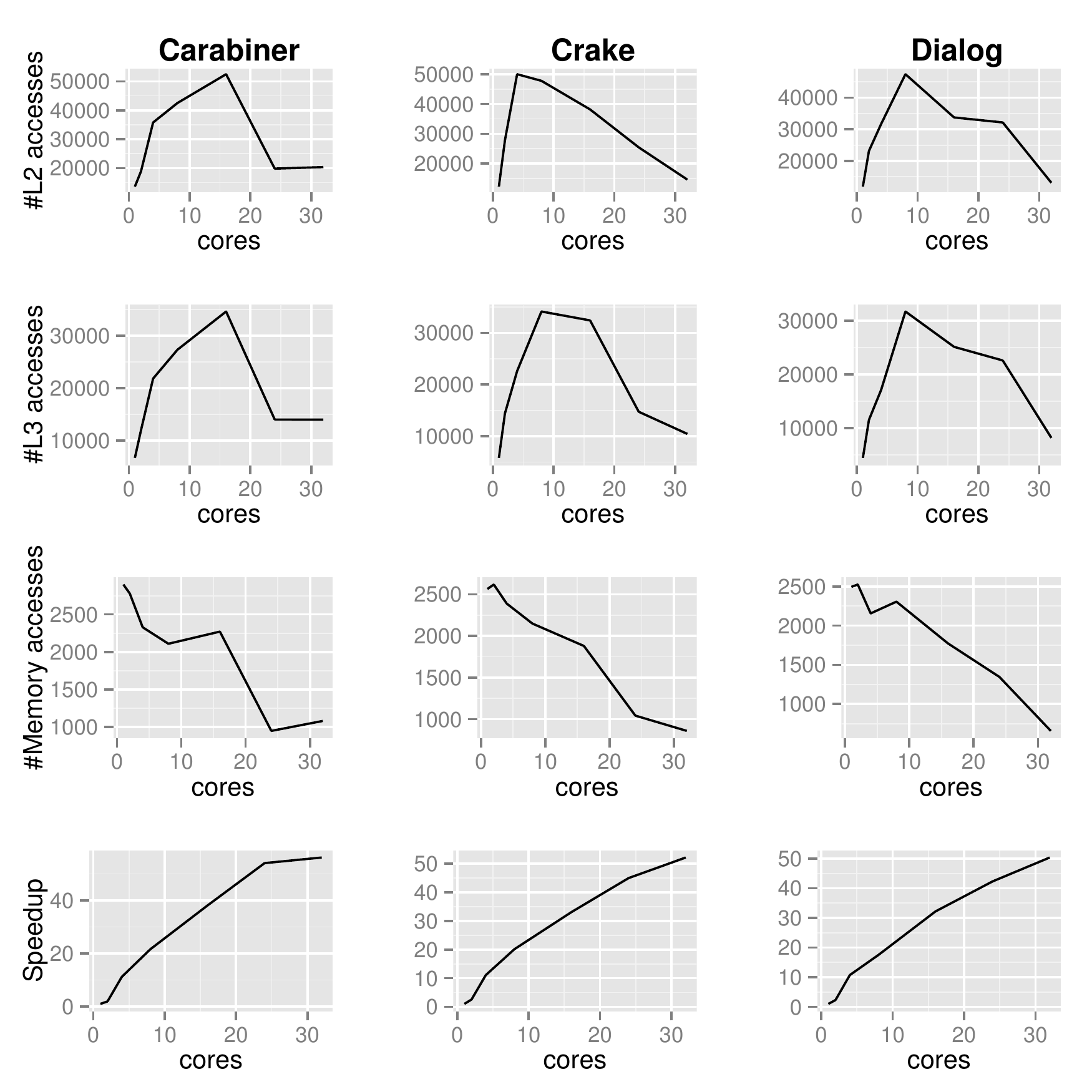}
  \caption{Number of accesses as a function of the number of cores used
for \ori for different meshes. Overall the distance where the data is fetched
decreases with the number of cores.}
\label{fig:scal_accesses}
\end{figure}
We can observe that with the scalability, the distance with respect of the data
accessed decreases. This could explain part of the superlinear speedup
observed, in particular, we can see that the fluctuations of the slopes of the
speedup of both Crake and Dialog seem to follow the fluctuations of the number
of memory accesses slope.
When we compare the speedup per ordering (that is: $T_{ord}(1)/T_{ord}(p)$),
both \bfs and \rdr also have the same superlinear speedup than \ori. Hence we
suspect that this superlinear speedup has more to do with the architecture (for
instance with less than four threads, they might be distributed in a
``scattered'' way, leading to four times the L3 caches from one to four cores)
or the LMS than with the reorderings (or lack of). The fact that (i) this
observation is independent of the ordering, and (ii) the speedup becomes more
``normal'' from 4 to 32 cores (about 7) increases the likeliness of this
hypothesis.

\begin{figure}
\centering
\includegraphics[width=0.8\linewidth]{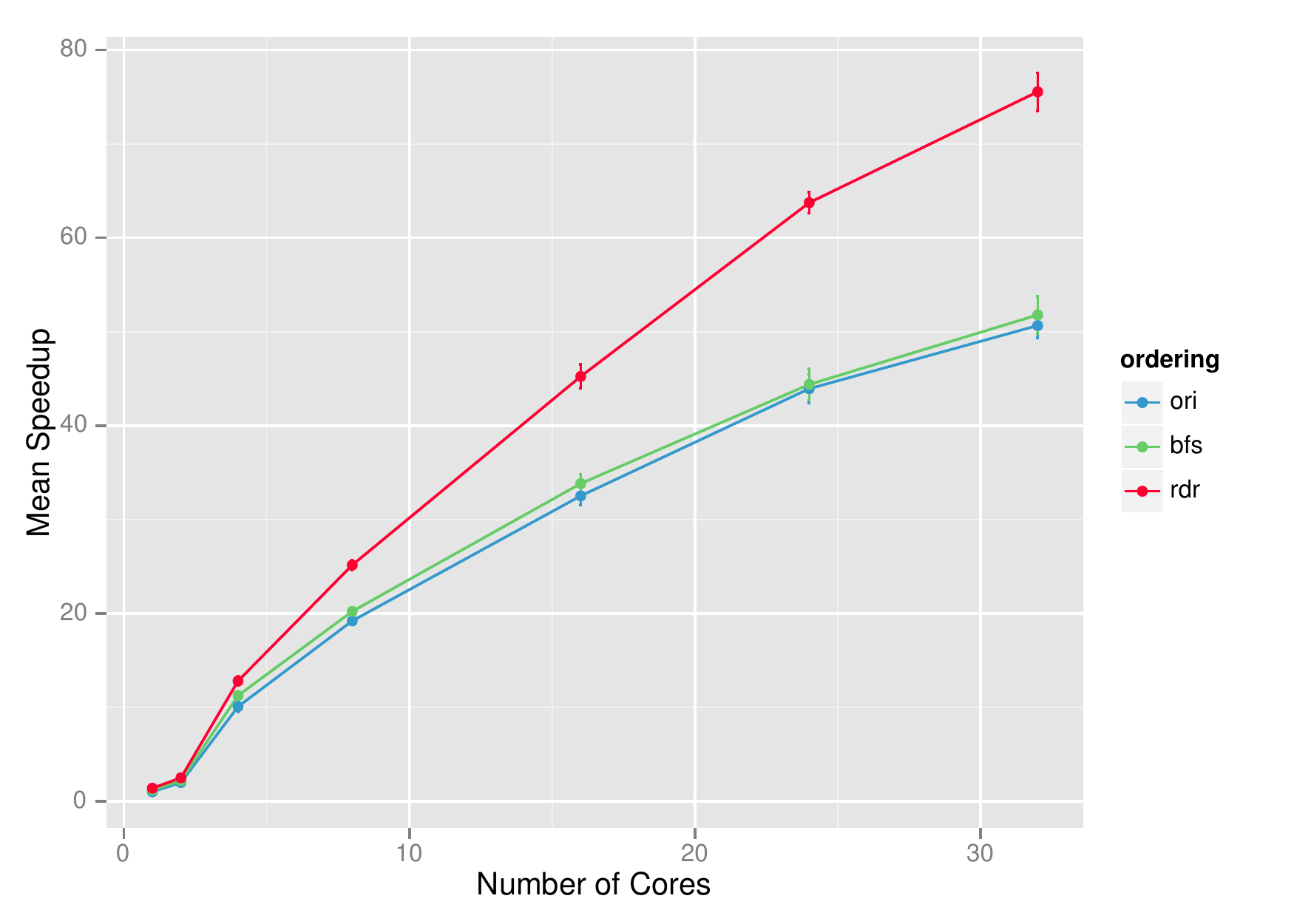}
  \caption{Mean speedup versus $T_{\ori}(1)$.\label{fig:scal_speedup}}
\vspace{-0.5cm}
\end{figure}

The second notable result is that the speed-up gained by our algorithm is 
greater than the \bfs ordering on almost all data sets. Furthermore, the average
speed-up is clearly dominant as can be seen in Figure~\ref{fig:scal_speedup},
with a speed-up greater than 75 when \rdr is applied on 32 cores!

\begin{figure}
\centering
\includegraphics[width=0.9\linewidth]{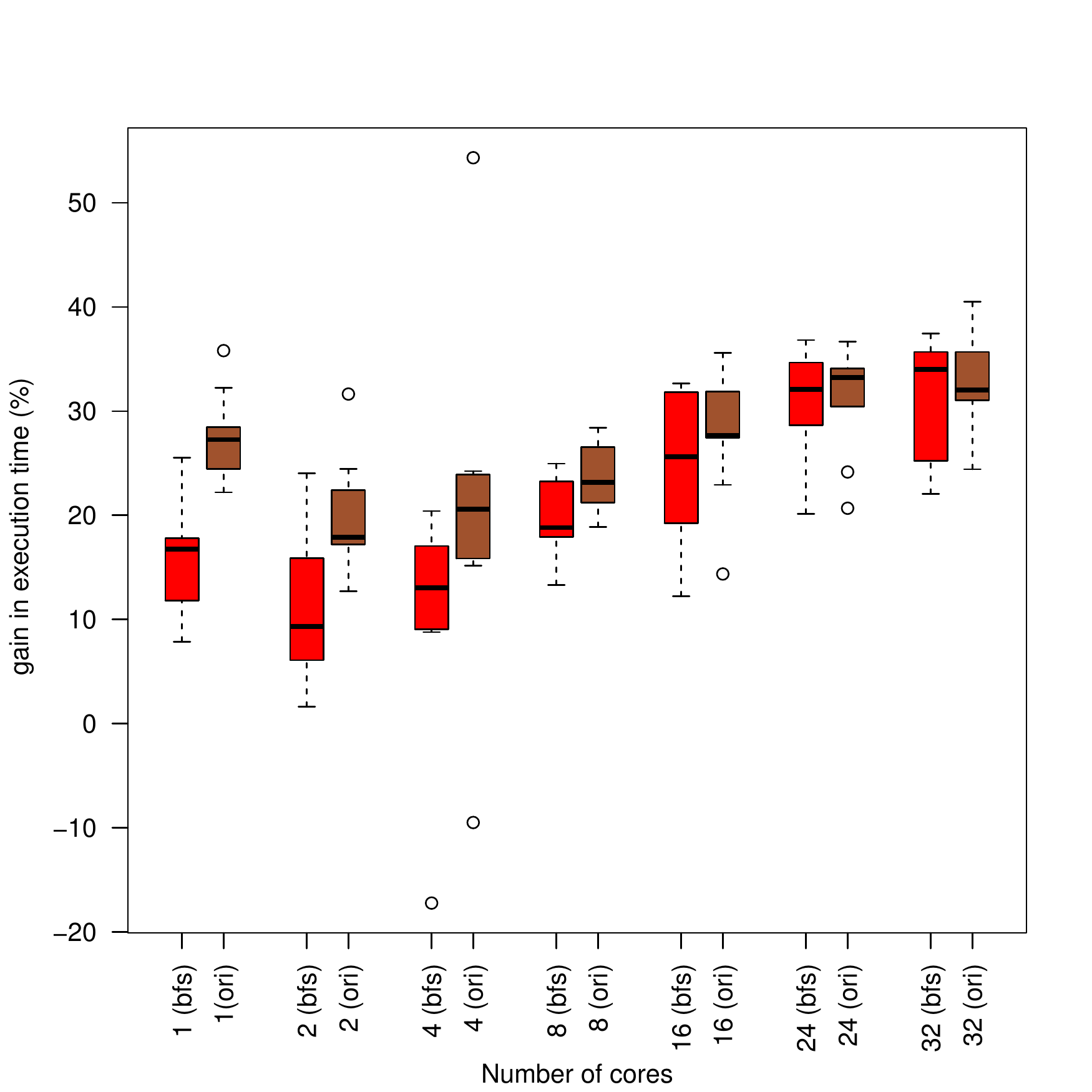}
  \caption{Gain with scalability the performance gain is
$\frac{T_{\texttt{algo}}(x)-T_{\rdr}(x)}{T_{\texttt{algo}}(x)}$, for
$\texttt{algo}$ being either \ori of \bfs and $x$ being the number of cores.\label{fig:gain}}
\vspace{-0.5cm}
\end{figure}

Finally, we plot on Figure~\ref{fig:gain} the gain in execution time of \rdr
compared to the execution time of \ori and \bfs. Clearly \rdr is dominant on
most of the meshes and not only on average. The only exception is the execution
of Valve on 4 cores. Overall, the expected gain over the \ori ordering varies
between 20\% and 30\% depending on the number of cores, and between 10\% and
30\% compared to the \bfs reordering.

\subsection{Discussion on reordering cost}

To conclude this experimental evaluation, we discuss the additional cost
incurred by the pre-computation (reordering). Because it follows so closely the
actual mesh smoothing algorithm, our reordering has a cost of approximatively
one iteration with the \ori ordering.

Hence, with an average gain between 20 and 30\% depending on the number of cores
(Figure~\ref{fig:gain}) over the computation with the \ori ordering, it follows
that the gain will be notable as soon as there are more than four iterations. To
conclude, we would not recommend doing any reordering when the initial mesh
quality is close to the desired mesh quality as one should not expect many
iterations. However if it is not the case, one should clearly use the
reordering we designed.

\section{Conclusion}
	\label{sec:conclusion}
	
In this work, we have presented a pre-computation heuristic for Laplacian mesh
smoothing. Our pre-computation takes the form of a reordering of the initial
data and performs remarkably well compared to a state of the art reordering
heuristic.

Our reordering is based on an observation we made, the reuse distance of the
mesh smoothing algorithm do not vary much over iterations. We hypothesized that
a good order for the first iteration would be efficient for subsequent
iterations.
We have developed and evaluate \rdr, a scheme that takes into account the
initial qualities of each vertex and reorders the mesh based on the quality in
ascending order to improve both temporal and spatial localities for memory
accesses of Laplacian mesh smoothing.

Thanks to our reordering, we decrease the number of cache misses of the
Laplacian Mesh Smoothing by 25\% (L1), 71\% (L2) and 84\% (L3) on average on a
single core. Similar cache performance is obtained when we run the application
on multicores (up to 32 cores). 
In turn, this decrease in cache miss allowed us to reach a mean speed-up of 75
when running on 32 cores compared to the single core performance with no
reordering, and a gain in execution time of 30\% compared to the 32-cores \ori
execution with as little as eight iterations of the mesh smoothing algorithm.
We were able to justify those gains by showing that our algorithm is
quasi-optimal with regard to the number of L2 and L3 misses.

By modifying almost nothing in the original algorithm, we have shown how much
cache-misses impact the execution time of the Laplacian mesh smoothing
algorithm. We expect our new reuse-distance-aware algorithm to outperform
extensions of Laplacian mesh smoothing as well.
We conjecture that either this ordering or an ordering based on the idea 
that it needs to be efficient for the first iteration, could improve other mesh
application performances such as mesh untangling~\cite{optimization2},
constraint mesh smoothing~\cite{optimization3}, and mesh
swapping~\cite{swapping}.

Finally, this result shows that data-locality is critical in the execution of 
applications, and that a short pre-computation may improve drastically the
execution of an application.

\section*{Acknowledgements}

Part of this work was done while the authors were at the Pennsylvania State
University. This material is based upon work supported by Vanderbilt University
and the National Science Foundation, CCF \#1319448.

\bibliographystyle{abbrv}
\bibliography{biblio}

\end{document}